\documentclass[12pt,reqno]{amsart}

\usepackage[english]{babel}
\usepackage[utf8]{inputenc}
\usepackage[T1]{fontenc}
\usepackage{amsmath,amssymb,amsthm,mathtools}
\usepackage{bbm}
\usepackage{enumitem}
\usepackage{aliascnt}
\usepackage[margin=2.6cm]{geometry}
\usepackage{comment}
\usepackage{graphicx}
\usepackage{csquotes}
\usepackage{microtype}
\usepackage[skip=.5\baselineskip]{parskip}
\usepackage{anyfontsize}

\usepackage[dvipsnames]{xcolor}
\definecolor{highlight}{cmyk}{90,99,0,0}

\usepackage[lining]{libertine}
\usepackage{courier}
\usepackage[libertine,libaltvw,liby]{newtxmath}
\makeatletter
    \renewcommand*\libertine@figurestyle{LF}
\makeatother

\setlist[enumerate,1]{label=(\roman*),itemsep=0.9ex}
\setlist[itemize]{itemsep=0.9ex}

\usepackage[colorlinks=true,allcolors=highlight,hyperfootnotes=false,linktocpage=true]{hyperref}
\usepackage{url}
\usepackage[nameinlink,capitalize]{cleveref}
\usepackage{tikz,tikz-cd}
\usepackage{pgfplots}
\pgfplotsset{compat=1.18}
\usepackage{subcaption}

\hypersetup{
  colorlinks=true,
  linkcolor=blue,
  citecolor=blue,
  urlcolor=blue,
  pdftitle={When Finite Free Curves Split},
  pdfauthor={Baran Hashemi and Jihoon Hyun}
}

\hypersetup{
  pdftitle={When Finite Free Curves Split},
  pdfauthor={Baran Hashemi and Jihoon Hyun}
}

\newtheorem{theorem}{Theorem}[section]
\newaliascnt{proposition}{theorem}
\newtheorem{proposition}[proposition]{Proposition}
\aliascntresetthe{proposition}
\newaliascnt{lemma}{theorem}
\newtheorem{lemma}[lemma]{Lemma}
\aliascntresetthe{lemma}
\newaliascnt{corollary}{theorem}
\newtheorem{corollary}[corollary]{Corollary}
\aliascntresetthe{corollary}
\newaliascnt{conjecture}{theorem}
\newtheorem{conjecture}[conjecture]{Conjecture}
\aliascntresetthe{conjecture}
\theoremstyle{definition}
\newaliascnt{definition}{theorem}

\aliascntresetthe{definition}
\newaliascnt{example}{theorem}
\newtheorem{example}[example]{Example}
\aliascntresetthe{example}
\newaliascnt{question}{theorem}
\newtheorem{question}[question]{Question}
\aliascntresetthe{question}
\theoremstyle{remark}
\newaliascnt{remark}{theorem}
\newtheorem{remark}[remark]{Remark}
\aliascntresetthe{remark}
\newtheorem*{remark*}{Remark}

\crefname{theorem}{theorem}{theorems}
\Crefname{theorem}{Theorem}{Theorems}
\crefname{proposition}{proposition}{propositions}
\Crefname{proposition}{Proposition}{Propositions}
\crefname{lemma}{lemma}{lemmas}
\Crefname{lemma}{Lemma}{Lemmas}
\crefname{corollary}{corollary}{corollaries}
\Crefname{corollary}{Corollary}{Corollaries}
\crefname{conjecture}{conjecture}{conjectures}
\Crefname{conjecture}{Conjecture}{Conjectures}
\crefname{question}{question}{questions}
\Crefname{question}{Question}{Questions}
\crefname{definition}{definition}{definitions}
\Crefname{definition}{Definition}{Definitions}
\crefname{example}{example}{examples}
\Crefname{example}{Example}{Examples}
\crefname{remark}{remark}{remarks}
\Crefname{remark}{Remark}{Remarks}

\newcommand{\R}{\mathbb R}
\newcommand{\C}{\mathbb C}
\newcommand{\one}{\mathbbm{1}_n}
\newcommand{\Var}{\operatorname{Var}}
\newcommand{\spec}{\operatorname{spec}}
\newcommand{\He}{\operatorname{He}}
\newcommand{\cN}{\mathcal N}
\newcommand{\cM}{\mathcal M}

\newcommand{\eps}{\varepsilon}

\title[When Finite Free Curves Split]{When Finite Free Curves Split}

\author{Baran Hashemi}
\address{Max Planck Institute for Mathematics in the Sciences, Leipzig, Germany}
\email{baran.hashemi@mis.mpg.de}

\author{Jihoon Hyun}
\address{Korea Advanced Institute of Science and Technology (KAIST),\newline Daejeon 34141, Republic of Korea}
\email{qawbecrdtey@kaist.ac.kr}

\begin{document}

\begin{abstract}
We characterize equality in the finite free Stam and entropy-power inequalities, proving that Hermite polynomials are the unique extremizers among simple real-rooted inputs, up to independent translations and scalings. The proof turns this classification into a rigidity problem for projective plane curves. Using hyperbolicity and the Helton-Vinnikov theorem, we express the Jacobian defect as an off-diagonal squared norm in a definite symmetric pencil. Together with score transport, this yields a matrix proof of Stam for all real-rooted inputs. For each simple base pair, the directions in which the defect vanishes are the independent translations and common dilation, forming a three-dimensional subspace in every degree whose curves split into $n$ projective lines. At a collision, the leading configurations are again finite free convolutions of normalized derivatives of the velocity polynomials of the colliding input clusters. Combined with incidence counting, this local formula shows that at most one real fiber is singular, every real singularity is an ordinary totally real multiple point, and the ordered collision multiplicities determine the real normalization covering. For $n\geq3$, every non-split curve has at least $2n-2$ non-real projective discriminant zeros, counted with multiplicity, with equality attained by irreducible curves of geometric genus zero through every simple input pair. The leading Fisher-information coefficient is determined by the colliding tangent configurations, while the finite entropy term retains the gaps between clusters. For $n\geq3$, maximal logarithmic entropy divergence along the optimally weighted score direction is equivalent to Stam equality.
\end{abstract}

\maketitle

\section{Introduction}\label{sec:introduction}

\subsection{Motivation}
Stam's inequality and the entropy-power inequality describe how information changes when independent random variables are added, and Gaussian distributions attain equality in both. Finite free probability replaces probability measures by monic real-rooted polynomials and addition of independent variables by a convolution of their root configurations. Garza--Vargas, Srivastava, and Stier
proved the corresponding inequalities and showed that translated
and scaled Hermite polynomials attain equality~\cite{GVS26}. We prove the converse when both inputs have simple real roots. For entropy power, this settles the uniqueness assertion in Gribinski's conjecture on the simple-rooted locus
\cite[Conjecture~1]{Gribinski19}.

The proof leads to a very interesting geometric question. Starting from
two simple root configurations, suppose that the input roots move
along straight lines. We ask when the output roots also follow
straight lines. Translations and a common dilation have this
property, and we prove that they are the only possibilities, regardless of the degree or the arrangement of the input roots. Using the same family of curves, we give an explicit description of partial collisions and a way to recognize information equality from an entropy singularity.

Let \(n\geq2\), and write
\(P_\alpha(x)=\prod_{i=1}^n(x-\alpha_i)\) for a polynomial with distinct
real roots.  Its score vector and finite free Fisher information are
\begin{equation}\label{eq:intro-fisher}
 s(\alpha)_i=\sum_{j\ne i}\frac{1}{\alpha_i-\alpha_j},
 \qquad
 \Phi_n(P_\alpha)=\frac{4}{n(n-1)^2}\|s(\alpha)\|_2^2.
\end{equation}
The score is the force exerted on each root by the other roots in the
logarithmic repulsion.  Its squared norm is a multiple of the sum of
the inverse squares of the root gaps.  The corresponding entropy and
entropy power are
\begin{equation}\label{eq:intro-entropy}
 \chi_n[P_\alpha]=\frac{2}{n(n-1)}
       \sum_{i<j}\log|\alpha_i-\alpha_j|,
 \qquad \cN_n(P_\alpha)=e^{2\chi_n[P_\alpha]}.
\end{equation}
At repeated roots we set \(\Phi_n=+\infty\) and \(\chi_n=-\infty\),
so reciprocal Fisher information and entropy power are both zero. The score vector itself is used only for simple configurations. Along the polynomial heat flow, each root moves with its score as
velocity. This is why the logarithmic interaction controls both information and motion.

Finite free additive convolution, denoted by \(\boxplus_n\), preserves
real-rootedness and adds root means and variances.  We recall its
coefficient formula in \Cref{sec:preliminaries}.  The operation goes back
to Walsh and was developed in its finite free form by Marcus, Spielman,
and Srivastava~\cite{Walsh22,MSS22}. Garza--Vargas, Srivastava, and Stier
proved the finite free Stam and entropy-power inequalities
\begin{align}
 \frac{1}{\Phi_n(f\boxplus_n g)}
 &\geq \frac{1}{\Phi_n(f)}+\frac{1}{\Phi_n(g)},
 \label{eq:intro-stam}\\
 \cN_n(f\boxplus_n g)
 &\geq \cN_n(f)+\cN_n(g).
 \label{eq:intro-epi}
\end{align}
Their proof relates the score vectors to the Jacobian of the convolution
root map and uses convexity of the roots of a hyperbolic polynomial
\cite{GVS26}. These inequalities also have counterparts in free probability.
Voiculescu developed the free Fisher-information theory, including
its convolution inequalities~\cite{Voiculescu98}, while Szarek and
Voiculescu proved the free entropy-power inequality \cite{SzarekVoiculescu96}. The equality problem considered here is on their finite polynomial counterparts.

The expected equality configurations are Hermite polynomials.  To fix the
normalization, let \(\He_n\) be the monic probabilists' Hermite polynomial,
which satisfies \(\He_n''-x\He_n'+n\He_n=0\), and put
\begin{equation}\label{eq:intro-heat-hermite}
 H_{n,\tau}(x)=\exp\!\left(-\frac{\tau}{2}\partial_x^2\right)x^n
             =\tau^{n/2}\He_n(x/\sqrt\tau),\qquad \tau>0.
\end{equation}
Their convolution satisfies
\(H_{n,\tau}\boxplus_nH_{n,\sigma}=H_{n,\tau+\sigma}\).
Consequently, independently translated Hermite polynomials of arbitrary
positive scales attain equality in both inequalities.  Gribinski had
conjectured the entropy-power inequality together with this uniqueness
statement~\cite[Conjecture~1]{Gribinski19}.  While the inequalities and the existence of Hermite solutions are established in~\cite{GVS26}, in this study we settle the ultimate question of uniqueness of these solutions. 

The classical analogy provides the evidence for the last step of an equality proof. An affine classical score characterizes a Gaussian density
\cite{Stam59,Blachman65,Carlen91} where for a centered root configuration,
the condition that \(s(\alpha)\) be proportional to \(\alpha\)
characterizes a scaled Hermite polynomial. However, the problem is to force this radial score from equality in convolution.

Let \(\Omega(\alpha,\beta)\) be the increasingly ordered output roots,
write \(J=D\Omega(\alpha,\beta)\), and put
\(I_\alpha=\|s(\alpha)\|_2^2\), with \(I_\beta\) defined
similarly.  Score transport makes Stam equality equivalent to norm
preservation on the direction
\begin{equation}\label{eq:intro-optimal-score-direction}
 w_* = I_\alpha^{-1}s(\alpha)\oplus I_\beta^{-1}s(\beta).
\end{equation}

Garza--Vargas, Srivastava, and Stier express the Jacobian defect as a weighted sum of the Hessians of the output roots. Convexity of ordered root sums, a general property of hyperbolic polynomials~\cite[Corollary~3.3]{BGLS01}, makes this weighted combination positive semidefinite. This proves that \(J\) is a
contraction on the separately centered tangent space \cite[Lemmas~5.1--5.2 and Proposition~5.3]{GVS26}. They also identify the common radial direction as a singular direction with singular value \(1\) \cite[Remark~5.4]{GVS26}.
To determine the equality cases, one must show that no other separately centered direction preserves norm.

We resolve this obstruction by following the roots as the inputs move. Keeping each input mean fixed, we let the roots travel along straight lines. Our matrix representation shows that norm preservation forces the output roots to do the same. We then prove that this can happen only when both input configurations undergo a common dilation. This geometric rigidity identifies the one-dimensional nullspace and leads to the Hermite equality cases.

The same curves also reveal what happens when only some roots collide. The striking feature is that their separation is governed by another finite free convolution, now of lower degree. This smaller convolution determines the leading behavior of Fisher information and entropy near the collision. Finally, we return to the motion prescribed by the weighted score \(w_*\). Along this motion, the logarithmic entropy divergence is maximal exactly when Stam is sharp. Equality at the starting configuration can therefore be recognized through a collision elsewhere on its curve.

\subsection{Main idea}
Let the roots of the two input polynomials move along straight lines, with velocity vectors \(u\) and \(v\). As we convolve the polynomials at each time, the output roots trace paths of their own. When are these paths also straight? Writing \(w=u\oplus v\), the Jacobian gives their initial velocity \(Jw\). To answer the question, we will study the plane curve traced by the output roots throughout the motion. We do so in the homogeneous polynomial
\begin{equation}\label{eq:intro-ternary-form}
 h_w(x,y,t)=\mathcal F_n(x,y\alpha+tu,y\beta+tv),
 \qquad
 \mathcal F_n(x,\zeta,\eta)
       =(P_\zeta\boxplus_nP_\eta)(x).
\end{equation}
We call \(C_w=V(h_w)\subset\mathbf P^2_{\R}\) the \emph{finite free curve} of the chosen direction.  In the chart \(y=1\), its points record the output roots at time \(t\).  Since both input polynomials are real-rooted for every real \((y,t)\), the form \(h_w\) is hyperbolic in the \(x\)-direction.  In three variables, the Helton--Vinnikov theorem gives a representation \(h_w(x,y,t)=\det(xI-yA_w-tB_w)\) with real symmetric matrices~\cite{HeltonVinnikov07,LewisParriloRamana05}. 

This representation allows us to compare two descriptions of the same
motion. In an orthonormal eigenbasis of \(A_w\), the diagonal entries of
\(B_w\) are the output velocities \(Jw\). Mean and variance additivity
determine the full Frobenius norm of \(B_w\).  If both input velocities
are centered, subtraction gives
\begin{equation}\label{eq:intro-centered-defect}
 \|w\|_2^2-\|Jw\|_2^2
     =\|\operatorname{off}_{A_w}(B_w)\|_{\mathrm F}^2.
\end{equation}
Here \(\operatorname{off}_{A_w}(B_w)\) is the off-diagonal part in
this eigenbasis. Equality forces the two symmetric matrices to
commute. They can therefore be diagonalized together, and \(h_w\)
splits into real linear factors. Each output branch is consequently
a straight line for all real times. Thus, a condition on the initial
velocity has become a statement about the entire curve. One direction at a time is enough. We may choose the pencil separately for each \(w\), and
\eqref{eq:intro-centered-defect} ensures that its off-diagonal
energy is independent of that choice.

\Cref{fig:pencil-defect-geometry} gives a degree-two example.
Both motions have the same initial output roots and tangents, while only the zero-defect motion follows those tangents for all times.

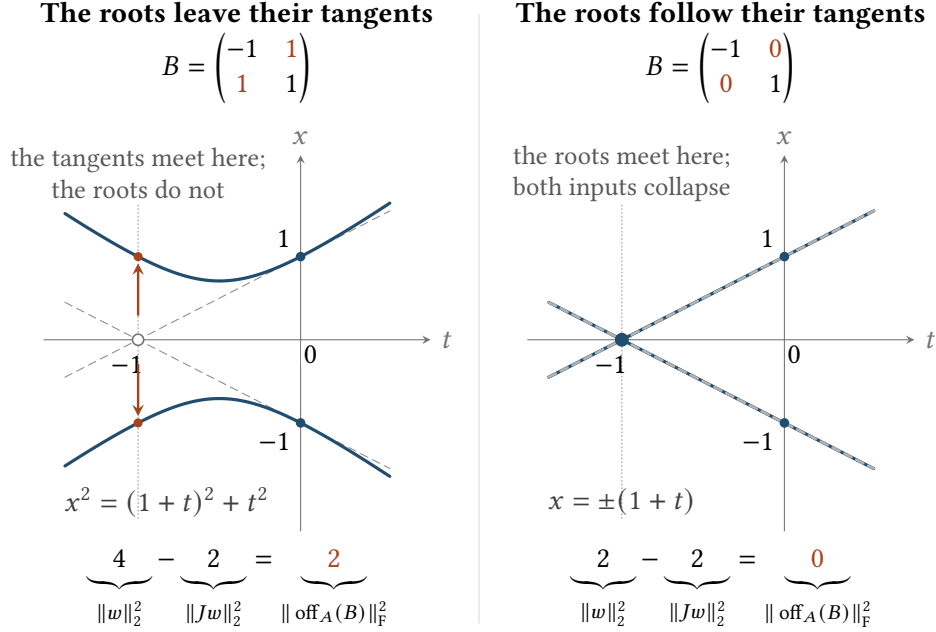
\begin{figure}[!htbp]
\centering
% Exact degree-two comparison. This file requires only amsmath and tikz.
\begingroup
\definecolor{ffRoot}{RGB}{33,77,110}
\definecolor{ffOff}{RGB}{167,67,28}
\begin{tikzpicture}[font=\small,>=stealth,line cap=round,line join=round]
  \path[use as bounding box] (0.15,-4.05) rectangle (12.55,4.62);
  \draw[black!15] (6.35,-3.75) -- (6.35,4.30);

  % ------------------------------------------------------------------ left
  \node[font=\small\bfseries] at (3.15,4.30) {The roots leave their tangents};
  \node at (3.15,3.60)
    {\(B=\begin{pmatrix}-1&\color{ffOff}1\\
                         \color{ffOff}1&1\end{pmatrix}\)};
  \begin{scope}[shift={(3.99,0)},x=2.15cm,y=1.10cm]
    \draw[->,black!55,thin] (-1.58,0) -- (0.80,0) node[right] {\(t\)};
    \draw[->,black!55,thin] (0,-2.30) -- (0,2.22) node[above] {\(x\)};
    \draw[black!40,densely dotted] (-1,-2.15) -- (-1,1.62);
    \draw[black!55] (-1,0.06) -- (-1,-0.06);
    \node[below left,inner sep=1pt] at (-0.95,-0.15) {\(-1\)};
    \node[below right,fill=white,inner sep=1pt] at (0,-0.02) {\(0\)};
    % tangents at t=0, then the curve on top
    \draw[black!45,densely dashed,thin,domain=-1.45:0.55,samples=2]
      plot (\x,{1+\x});
    \draw[black!45,densely dashed,thin,domain=-1.45:0.55,samples=2]
      plot (\x,{-1-\x});
    \draw[ffRoot,very thick,domain=-1.45:0.55,samples=201]
      plot (\x,{sqrt((1+\x)^2+(\x)^2)});
    \draw[ffRoot,very thick,domain=-1.45:0.55,samples=201]
      plot (\x,{-sqrt((1+\x)^2+(\x)^2)});
    % the gap that never closes
    \draw[ffOff,line width=0.9pt,->] (-1,0.30) -- (-1,0.92);
    \draw[ffOff,line width=0.9pt,->] (-1,-0.30) -- (-1,-0.92);
    \fill[ffOff] (-1,1) circle[radius=1.7pt];
    \fill[ffOff] (-1,-1) circle[radius=1.7pt];
    \filldraw[fill=white,draw=black!55,line width=0.6pt]
      (-1,0) circle[radius=2.1pt];
    \node[align=center,text=black!65,font=\footnotesize] at (-1,2.00)
      {the tangents meet here;\\the roots do not};
    \fill[ffRoot] (0,1) circle[radius=1.8pt];
    \fill[ffRoot] (0,-1) circle[radius=1.8pt];
    \node[left,inner sep=1.5pt] at (-0.035,1.22) {\(1\)};
    \node[left,inner sep=1.5pt] at (-0.035,-1.22) {\(-1\)};
    \node[right,text=black!75,inner sep=1pt] at (-1.47,-1.95)
      {\(x^2=(1+t)^2+t^2\)};
  \end{scope}
  \node at (3.15,-3.30)
    {\(\underbrace{4}_{\|w\|_2^2}
       -\underbrace{2}_{\|Jw\|_2^2}
       =\underbrace{\color{ffOff}2}_{\|\operatorname{off}_A(B)\|_{\mathrm F}^2}\)};

  % ----------------------------------------------------------------- right
  \node[font=\small\bfseries] at (9.55,4.30) {The roots follow their tangents};
  \node at (9.55,3.60)
    {\(B=\begin{pmatrix}-1&\color{ffOff}0\\
                         \color{ffOff}0&1\end{pmatrix}\)};
  \begin{scope}[shift={(10.39,0)},x=2.15cm,y=1.10cm]
    \draw[->,black!55,thin] (-1.58,0) -- (0.80,0) node[right] {\(t\)};
    \draw[->,black!55,thin] (0,-2.30) -- (0,2.22) node[above] {\(x\)};
    \draw[black!40,densely dotted] (-1,-2.15) -- (-1,1.62);
    \draw[black!55] (-1,0.06) -- (-1,-0.06);
    \node[below left,inner sep=1pt] at (-0.95,-0.15) {\(-1\)};
    \node[below right,fill=white,inner sep=1pt] at (0,-0.02) {\(0\)};
    % curve first, then the tangents on top: they coincide
    \draw[ffRoot,very thick,domain=-1.45:0.55,samples=2]
      plot (\x,{1+\x});
    \draw[ffRoot,very thick,domain=-1.45:0.55,samples=2]
      plot (\x,{-1-\x});
    \draw[black!30,densely dashed,line width=1.0pt,domain=-1.45:0.55,samples=2]
      plot (\x,{1+\x});
    \draw[black!30,densely dashed,line width=1.0pt,domain=-1.45:0.55,samples=2]
      plot (\x,{-1-\x});
    \fill[ffRoot] (0,1) circle[radius=1.8pt];
    \fill[ffRoot] (0,-1) circle[radius=1.8pt];
    \fill[ffRoot] (-1,0) circle[radius=2.6pt];
    \node[align=center,text=black!65,font=\footnotesize] at (-1,2.00)
      {the roots meet here;\\both inputs collapse};
    \node[left,inner sep=1.5pt] at (-0.035,1.22) {\(1\)};
    \node[left,inner sep=1.5pt] at (-0.035,-1.22) {\(-1\)};
    \node[right,text=black!75,inner sep=1pt] at (-1.47,-1.95)
      {\(x=\pm(1+t)\)};
  \end{scope}
  \node at (9.55,-3.30)
    {\(\underbrace{2}_{\|w\|_2^2}
       -\underbrace{2}_{\|Jw\|_2^2}
       =\underbrace{\color{ffOff}0}_{\|\operatorname{off}_A(B)\|_{\mathrm F}^2}\)};
\end{tikzpicture}%
\endgroup%
\caption{The matrix defect in degree two.  Both panels use
\(\alpha=\beta=(-1/\sqrt2,1/\sqrt2)\), and have the same output
roots and initial velocities at \(t=0\). The input velocities are
\(u=(0,0)\), \(v=(-\sqrt2,\sqrt2)\) on the left and
\(u=v=\alpha\) on the right.  The curves are drawn from the exact
determinant equations shown.  Dashed lines are the initial tangents and on the right they coincide with the whole curve.  The two affine branches intersect at \(t=-1\), where the increasingly ordered roots exchange branches. These are specializations of
\Cref{ex:quadratic-pencil}.}
\label{fig:pencil-defect-geometry}
\end{figure}

The pencil has now done all it can. It says the curve is a union of $n$ real lines, but it says nothing about which input motions produced them, because nothing in the matrix argument sees the inputs. For that we count intersections. The remaining argument concerns intersections of these affine root lines. Fujie's multiplicity theorem says that an output collision can occur only
if sufficiently many input roots collide~\cite[Propositions~3.6
and~3.8]{Fujie26}. Two different output collision times would require
more input lines than are available, because two distinct affine lines can
intersect at most once.  If the output slopes are not all equal, the incidence
argument forces every output line through one common point.  At that time the output variance is zero and variance additivity forces both inputs to collapse as well. After their translations are removed, the two input motions must therefore be the same dilation. That is Theorem 1.2, and it is why the split locus stays three-dimensional however large $n$ becomes.

Everything so far has assumed that the curve splits into lines, which is the case the equality theorem needs and the exception rather than the rule. The incidence count only rules out two distinct real collision parameters for every curve in the family, split or not. So there is always at most one interesting fiber. What does the curve look like there, when only some of the output roots meet?

To answer that we return to the permutation formula for convolution. Suppose an input root of multiplicity $m$ and one of multiplicity $k$ produce an output root of multiplicity $r=m+k-n>0$. After magnifying the roots at that point, only permutations with exactly $r$ matches between the two input clusters survive. Averaging the surviving subsets differentiates the two velocity polynomials down to degree $r$ and averaging their matchings gives a degree-$r$ finite free convolution. So the local model of $\boxplus_n$ at a collision is $\boxplus_r$ and the operation is its own blow-up model.

The tangent polynomial supplies more than a multiplicity count.
For affine curves through a simple base pair, the normalized
derivatives have simple real roots, and so does their convolution.
These roots are the slopes of the branches through the singular
point. They also determine the leading Fisher coefficient, while
the number of colliding pairs determines the logarithmic entropy
coefficient. Normalization separates these branches, allowing us to follow
them around the real projective parameter line. Crossing the
exceptional fiber reverses each colliding block of ordered roots,
and the projective identification reverses the whole list.
Those two permutations determine the real covering.

\subsection{Main results}
The results come in two groups, both about the curves \(C_w\).
The first characterizes complete splitting and turns that classification
into the equality cases of Stam and entropy power.  The second describes
partial collisions, when only some of the output roots meet.
For a root vector \(\rho\), write
\(\bar\rho=n^{-1}\sum_i\rho_i\) and
\(\rho^\circ=\rho-\bar\rho\one\), where
\(\one\) is the all-ones vector.  The separately centered tangent space
is \(\mathcal V=\one^\perp\oplus\one^\perp\).  All norms on root
vectors are Euclidean unless another exponent is indicated.

\begin{theorem}[The matrix defect]\label{thm:intro-matrix-defect}
Let \(n\geq2\), and let \(\alpha,\beta\in\R^n\) be real root vectors such that
\(P_\alpha\boxplus_nP_\beta\) has simple roots.  Write
\(\gamma=\Omega(\alpha,\beta)\) and \(J=D\Omega(\alpha,\beta)\),
using the increasing order for the output roots.
For every \(w=u\oplus v\in\R^{2n}\), there are real symmetric
\(n\times n\) matrices \(A_w,B_w\) such that
\begin{equation}\label{eq:intro-definite-pencil}
 h_w(x,y,t)=\det(xI-yA_w-tB_w),
 \qquad \spec(A_w)=\{\gamma_1,\ldots,\gamma_n\}.
\end{equation}
For every such representation,
\begin{equation}\label{eq:intro-matrix-defect}
 \|w\|_2^2-\|Jw\|_2^2+2n\bar u\bar v
       =\|\operatorname{off}_{A_w}(B_w)\|_{\mathrm F}^2.
\end{equation}
The right side vanishes if and only if \(h_w\) is a product of \(n\)
real linear forms.
\end{theorem}

Only the output is required to be simple here; either simple input
suffices, by Fujie's theorem.  On \(\mathcal V\), the identity proves
that \(J\) is a contraction.  Score transport and the elementary
boundary cases then give a matrix proof of Stam for all real-rooted
inputs (\Cref{cor:matrix-stam}).  This answers the matrix-proof question
raised in~\cite[\S5, footnote~2]{GVS26}.

Classifying the directions on which the defect vanishes does require
both inputs to be simple.  This is a necessary hypothesis for the
classification, as the example in \Cref{rem:intro-simplicity-obstruction}
will show.

\begin{theorem}[Classification of split finite free curves]
\label{thm:intro-split-locus}
For simple real inputs \(\alpha,\beta\), the form \(h_{u\oplus v}\)
splits into \(n\) real linear factors if and only if there exist
\(a,b,c\in\R\) such that
\begin{equation}\label{eq:intro-split-directions}
 u=a\one+c\alpha^\circ,
 \qquad v=b\one+c\beta^\circ.
\end{equation}
\end{theorem}

\begin{remark*}
We can produce straight output branches by moving the inputs without changing their shapes: translate each root configuration separately, and rescale both centered configurations by the same factor.
\end{remark*}

These two translations and the common dilation span the \emph{split locus}
\begin{equation}\label{eq:intro-split-locus}
 L_{\alpha,\beta}
 =\operatorname{span}\{\one\oplus0,\ 0\oplus\one,\
                       \alpha^\circ\oplus\beta^\circ\}.
\end{equation}
Translation and dilation covariance of convolution make every direction
in this space split the curve. The theorem is the converse, stating that straight output branches force the inputs to move in exactly this way.  Convolution mixes the two configurations, but it cannot conceal a change in their shapes or a difference between their dilation rates.

The split locus is therefore 3-dimensional however large the degree. Raising \(n\) adds input velocities and creates many more ways for output
lines to meet.  Yet reflection symmetries, rational relations among root
gaps, and roots arbitrarily close to a collision produce no extra
splitting direction, as long as both inputs remain simple.  The incidence
argument rules out these possibilities by forcing both inputs to collapse
in one projective fiber and for pure translations, that fiber is at infinity.

The rigidity of the split locus can also be measured against
a naive dimension count. Products of \(n\) linear forms have codimension \(\binom n2\) in the projective space of ternary degree-\(n\) forms.  For the \(2n\)-dimensional space of input directions, the transverse prediction for the inverse image is therefore \(2n-\binom n2\).
For \(n\geq6\), this number is negative and a transverse family could
contain no split form.  Our inverse image has dimension three in every
degree, an excess of \(\binom{n-2}{2}\) over that prediction
(\Cref{prop:split-excess}). This comparison measures how far the finite
free family is from general position and it makes no assertion about the
complex inverse-image scheme.

\begin{corollary}[The unit singular direction]
\label{cor:intro-jacobian-rigidity}
Let \(\alpha,\beta\in\R^n\) both have simple coordinates, and let
\(J=D\Omega(\alpha,\beta)\). The restriction \(J|_{\mathcal V}\)
has operator norm \(1\), and its singular value \(1\) is simple.
More precisely,
\begin{equation}\label{eq:intro-unit-singular-space}
 \ker\!\left(I_{\mathcal V}-(J|_{\mathcal V})^*(J|_{\mathcal V})\right)
       =\operatorname{span}\{\alpha^\circ\oplus\beta^\circ\}.
\end{equation}
\end{corollary}

\begin{remark}
\label{rem:intro-simplicity-obstruction}
The distinction between the two simplicity hypotheses is important.
If \(P_\beta=x^n\), the output is \(P_\alpha\), and the matrix
defect identity still holds for simple \(\alpha\).  But every
centered direction \(u\oplus0\) now preserves norm.  These directions
form a space of dimension \(n-1\), which is larger than a line when
\(n\geq3\).  In degree two this space is already a line, so the
example does not distinguish the hypotheses.
\end{remark}

We now take \(w\) to be the weighted score direction \(w_*\) from
\eqref{eq:intro-optimal-score-direction}.  Score transport identifies
Stam equality with norm preservation on this direction.
\Cref{cor:intro-jacobian-rigidity} then forces \(w_*\) into the radial
line, so each score is proportional to its own centered root vector.
This rootwise identity gives a second-order differential equation for
each input polynomial.  After translation and positive rescaling, it is
the Hermite equation, whose monic degree-\(n\) solution is \(\He_n\).

For entropy power, we first normalize the two inputs separately
to unit entropy power and apply the heat-flow interpolation for
entropy concavity. Equality forces the weighted Fisher defect to
vanish along the interpolation. Continuity at the simple endpoint
inputs then gives the same Hermite classification. This yields
the following main equality theorem, proved in
\Cref{sec:information-equality}.

\begin{theorem}[Hermite equality]\label{thm:information-equality}
\label{thm:intro-stam-equality}\label{thm:intro-epi-equality}
Let \(f,g\) be monic degree-\(n\) polynomials with simple real roots.
Equality holds in the finite free Stam inequality
\eqref{eq:intro-stam} if and only if \(f\) and \(g\) are independently
translated and positively scaled copies of \(\He_n\).  The same
classification holds for equality in the entropy-power inequality
\eqref{eq:intro-epi}.
\end{theorem}

Positive scaling here means dilation of the roots: for \(c>0\), we write
\(c_*f(x)=c^nf(x/c)\).  The scales of the two equality inputs need not
agree.  The weighted Fisher inequality used in the entropy proof does
force equal variance (\Cref{cor:weighted-fisher-equality}) and that condition applies after the inputs have been rescaled separately to unit entropy power. Undoing those two normalizations permits independent scales in the final statement.

In degree two, the classification is automatic, since every simple
monic quadratic is \((x-m)^2-d^2=d^2\He_2((x-m)/d)\) for some
\(m\in\R\) and \(d>0\).  If entropy power is extended by zero to repeated-root polynomials, \(x^n\boxplus_n g=g\) gives boundary equalities for arbitrary \(g\). The entropy-power assertion settles the uniqueness part of Gribinski's conjecture~\cite[Conjecture~1]{Gribinski19} on the
finite-information locus, while the inequality itself was proved
in~\cite{GVS26}.

Convolution also contracts distances between input configurations,
strictly so unless the fixed factor is a point mass.  For increasingly
ordered root vectors, let
\(d_2(\rho,\sigma)=\|\rho^\circ-\sigma^\circ\|_2\).
For their centered empirical measures, this is \(\sqrt n\) times the
quadratic Wasserstein distance.  We prove that, for every fixed real-rooted
\(q\) with at least two distinct roots and any two simple inputs
\(\alpha,\widetilde\alpha\) distinct modulo translation,
\begin{equation}\label{eq:intro-strict-W2}
 d_2\bigl(\Omega(\alpha,\beta),
          \Omega(\widetilde\alpha,\beta)\bigr)
       <d_2(\alpha,\widetilde\alpha),\qquad q=P_\beta.
\end{equation}
The assumption on \(q\) is sharp, meaning that a point mass only translates the roots (\Cref{thm:strict-root-transport}).
For simple \(q\), strictness follows from the split classification.
To allow repeated roots in \(q\), we show instead that every entry of
the one-sided root Jacobian is strictly positive
(\Cref{thm:positive-one-sided-jacobian}).  Every output score is then an
average using every input score with positive weight, so strict convexity
gives strict contraction of every \(\ell^p\) score norm for
\(1<p<\infty\).  Thus the Fisher monotonicity of~\cite{GVS26}
extends to these score norms, with equality precisely when \(q\) is a
point mass (\Cref{thm:strict-score-monotonicity}).

Complete splitting is only one way for a finite free curve to degenerate.
We now turn to partial collisions.  When several output roots meet, in
which directions do they separate?

Multiplicity cannot answer that. Fujie's theorem determines how many
output roots meet, but it does not specify their separation velocities.
We compute this first-order profile.  In a local affine parameter
\(\varepsilon\) vanishing at a collision, let \(U,V\) be the monic
polynomials whose roots are the velocities within the two maximal colliding input blocks, of sizes \(m,k\).  Suppose \(r=m+k-n>0\).
After subtracting the collision location and dividing by \(\varepsilon\),
the \(r\) output roots converge, as a multiset, to the roots of
\begin{equation}\label{eq:intro-tangent-convolution}
 Q=
 \left(\frac{r!}{m!}U^{(m-r)}\right)
 \boxplus_r
 \left(\frac{r!}{k!}V^{(k-r)}\right),
 \qquad r=m+k-n>0.
\end{equation}

Both normalized derivatives have degree \(r\). Thus the first-order
separation of the colliding output roots is governed by a
degree-\(r\) finite free convolution of the differentiated input
velocity polynomials. \Cref{thm:tangent-convolution} proves this
formula for repeated base configurations and arbitrary affine
velocities, and gives the exact leading coefficient of the
rescaled polynomial. If \(Q\) is simple, its roots are the distinct
tangent slopes. This simplicity holds automatically for the curves
through a simple base pair considered below. Combined with the incidence argument, the tangent formula determines the real singularities and the covering of the normalized curve, as well as the leading singular terms of Fisher
information and entropy.

We now fix simple base inputs \(\alpha,\beta\) and return to the
curves \(C_w\).  For \(S\subset\{1,\ldots,2n\}\), let
\(\R^S=\operatorname{span}\{e_j:j\in S\}\), where the \(e_j\)
are the standard coordinate vectors of \(\R^{2n}\).  For
\(2\leq r\leq n\), write \(\Sigma_r\) for the directions \(w\)
for which some real projective fiber of \(h_w\) has a root of
multiplicity at least \(r\).  This multiplicity condition gives no
obvious reason for a linear description of \(\Sigma_r\).  Yet it is
a finite union of linear subspaces, all of the same dimension, which
the next theorem identifies exactly.  Its proof is in
\Cref{sec:collision-geometry}.

\begin{theorem}[Real singularities of finite free curves]
\label{thm:intro-real-geometry}
Fix \(n\geq2\) and simple real root vectors \(\alpha,\beta\).
Every real singularity of \(C_w\) is an ordinary multiple point with
distinct real tangent lines, and all real singularities lie in at most
one fiber of \([x:y:t]\mapsto[y:t]\).  Moreover,
\begin{equation}\label{eq:intro-collision-arrangement}
 \Sigma_r=
 \bigcup_{\substack{S\subset\{1,\ldots,2n\}\\|S|=n-r}}
 \bigl(L_{\alpha,\beta}+\R^S\bigr),
 \qquad 2\leq r\leq n.
\end{equation}
These are \(\binom{2n}{n-r}\) distinct linear subspaces of dimension
\(n-r+3\).  In particular, \(\Sigma_2\) is the locus of directions
giving a real singularity, and \(\Sigma_n=L_{\alpha,\beta}\).
\end{theorem}

An ordinary point of multiplicity \(r\) has \(r\) smooth branches
with pairwise distinct tangents, so the theorem rules out real
cusps and tangential intersections. For a simple base pair,
the polynomial \(Q\) in \eqref{eq:intro-tangent-convolution}
has \(r\) distinct real roots, and these roots give the slopes
of the branches.

The arrangement comes from the incidence count.  It places all real
root collisions in one fiber and permits at most \(n-r\) input coordinates
to escape a common collapse at that projective parameter.  For \(r=n\),
none escape, which is why \(\Sigma_n=L_{\alpha,\beta}\). Each step down in multiplicity permits one more exceptional input root. The restriction to real singularities here is important as \Cref{ex:reducible-real-smooth-finite-free-curve} gives a reducible curve with no real singular point.

Once the branches are separated by normalization, they can be followed
through a collision. Within each colliding group their order reverses
because the local parameter changes sign, while the projective
identification reverses the whole root list.  Let \((r_1,\ldots,r_{\kappa})\) be the multiplicities in the exceptional
fiber, including singletons, listed in the increasing order of the roots
themselves.  The covering is given by a simple rule. Reverse each consecutive block of these lengths, then reverse all
\(n\) positions.

The cycle lengths of the resulting permutation are the covering degrees
of the connected components of the real normalization over the real
projective parameter line (\Cref{thm:real-monodromy}).  If there is no real collision, the first reversal is the identity.  The order of the multiplicities is important here and not merely the multiset meaning that two collisions with the same multiplicities in different positions give different coverings. In degree four, \((2,1,1)\) gives a single component covering four times, while \((1,2,1)\) gives three components, of degrees \(2,1,1\). The real-fiberedness and absence of real ramification also follow from Kummer--Shamovich's general results, discussed in \Cref{sec:intro-discussion}. The finite free structure adds the explicit permutation.

Every ordered composition occurs through every fixed simple base pair
(\Cref{cor:collision-composition-realization}).  For \(n\geq3\), a fiber
with exactly two distinct output roots of coprime multiplicities gives
a permutation with a single cycle, forcing absolute irreducibility
(\Cref{cor:two-block-collision-irreducibility}).

The same branch expansion counts discriminant zeros.  A block of size
\(r_j\) contributes order \(r_j(r_j-1)\), because each of its root
gaps vanishes to first order.  The total discriminant degree is
\(n(n-1)\), but a non-split curve can account for at most
\((n-1)(n-2)\) of that degree at its real collision fiber.
Thus, for \(n\geq3\), every non-split curve has at least \(2n-2\)
non-real projective zeros of its \(x\)-discriminant, counted with
multiplicity.  The bound is attained exactly on
\(\Sigma_{n-1}\setminus L_{\alpha,\beta}\)
(\Cref{cor:non-real-discriminant-gap}), these curves are rational, with
one ordinary \((n-1)\)-fold point as their only complex singularity.
In this extremal case the remaining discriminant zeros therefore record
only non-real ramification.  Their total multiplicity \(2n-2\) is
exactly the ramification degree given by Riemann--Hurwitz for the
degree-\(n\) projection \(\mathbf P^1\to\mathbf P^1\).

The tangent polynomial also determines how information diverges.
The shrinking gaps make Fisher information diverge through their
reciprocal squares, while their logarithms govern the entropy divergence.
If precisely one output cluster of size \(r\geq2\) collides and its tangent polynomial \(Q\) is simple, the output polynomial \(R_\varepsilon\) at local parameter
\(\varepsilon\) satisfies
\begin{align}\label{eq:intro-collision-information}
 \lim_{\varepsilon\to0}\varepsilon^2\Phi_n(R_\varepsilon)
 &=\frac{r(r-1)^2}{n(n-1)^2}\Phi_r(Q),\notag\\
 \chi_n[R_\varepsilon]
 &=\frac{r(r-1)}{n(n-1)}\log|\varepsilon|+O(1).
\end{align}
The Fisher coefficient remembers the separation velocities through
\(Q\), whereas the logarithmic entropy coefficient counts colliding
pairs: \(\binom r2\) out of the \(\binom n2\) output-root pairs.
For a double collision in degree three, it is \(1/3\).
\Cref{cor:collision-information} gives the exact entropy constant and
adds the contributions of all clusters when several collide.

This brings us back to Stam equality.  Let \(\vartheta(w)\) be the
proportion of output-root pairs that intersect in the exceptional fiber.
Equivalently, it is the total real projective zero order of
\(\operatorname{Disc}_x h_w\), divided by \(n(n-1)\), set
\(\vartheta(w)=0\) when there is no real collision.  In a local
affine chart at the collision, it is the coefficient of
\(\log|\varepsilon|\) in the entropy.  This includes a collision
at infinity in the original time chart.

Every input pair admits a common dilation, so complete collapse along
an arbitrary direction cannot distinguish Hermite inputs.  We must use
the direction selected by score transport.  For \(n\geq3\),
\Cref{cor:score-entropy-equality} gives
\begin{equation}\label{eq:intro-entropy-equality-criterion}
 \vartheta(w_*)=1
 \quad\Longleftrightarrow\quad
 \text{equality in Stam for }P_\alpha,P_\beta,
 \qquad
 w_*=I_\alpha^{-1}s(\alpha)\oplus I_\beta^{-1}s(\beta).
\end{equation}
Every other simple input pair has \(\vartheta(w_*)\leq(n-2)/n\).

Thus an equality condition measured at the simple base fiber is visible
as an entropy singularity elsewhere on the same curve.  The matrix defect
sees it at \(t=0\), the tangent convolution sees it at the collision.
They are two readings of one condition.

\subsection{Discussion}
\label{sec:intro-discussion}
Leake and Ryder used the hyperbolicity of the universal convolution
polynomial to derive root inequalities and majorization
results~\cite[Section~1.1]{LeakeRyder18}. Here, we provide a new bridge between real algebraic geometry to finite free information at two different stages. A definite pencil expresses the Jacobian defect as a squared off-diagonal norm. Then, the incidence argument classifies the input motions for which this norm vanishes. At partial collisions, the tangent convolution
determines the branch slopes and establishes the ordinary
singularity type. Together, these arguments describe the splitting
and real singularity loci directly in the input velocities.

Our analytic starting point is the work of Garza--Vargas, Srivastava,
and Stier~\cite{GVS26}. They proved both information inequalities,
and we use their score-transport identity and de Bruijn formula.
\Cref{cor:matrix-stam} answers their request for a matrix-analytic
proof of Stam, for all real-rooted inputs. Their Remark~5.4
identifies the common radial direction as a unit singular direction
of the convolution Jacobian. For simple inputs,
\Cref{cor:intro-jacobian-rigidity} shows that this is the only
norm-preserving direction after the two translations are removed.
This supplies the missing step from sharpness at Hermite inputs
to the classification of equality.

The closest multiplicity result is Fujie's atom theorem \cite[Theorem~1.1 and Propositions~3.6--3.8]{Fujie26}.  It identifies the
locations and multiplicities of repeated output roots.  Earlier work of
Kostov and Shapiro on Schur--Szeg\H{o} composition also studies
multiplicities and qualitative splitting under perturbation
\cite[Proposition~1.4 and Theorem~1.6]{KostovShapiro06}.  Our additive
tangent formula \eqref{eq:intro-tangent-convolution} gives the
separation velocities themselves.  It can therefore be used whenever
a degeneration is known through its input root velocities, without
first solving for the nearby output roots.  This provides an explicit
local answer to the root-spacing question that opens \cite{GVS26}.

The projection of the normalized curve belongs to the theory of
real-fibered morphisms.  Its real-fiberedness and its lack of real
ramification also follow from Kummer and Shamovich
\cite[Corollary~2.16 and Theorem~2.19]{KummerShamovich20}.
What the finite free structure provides is the ordered block reversal
that computes the covering, together with an explicit realization of
every such type through every simple base pair.  On each irreducible
component, the distinguished projection gives a degree vector in the
separating semigroup of Kummer and Shaw~\cite{KummerShaw20}.
Determining the other separating morphisms is a further question about
that curve. The non-real discriminant bound follows from the incidence and local branch calculations. In the equality case, the genus formula proves rationality, while the incidence classification describes the extremal locus directly in the input velocities. The relation between real-fibered curves and definite determinantal representations also appears in the algebraic constructions of
Hanselka and Kummer \cite[Corollaries~10.3--10.4]{HanselkaKummer24}.

The appendices examine the regular geometry near equality and the
algebraic meaning of the classified loci.  The first author's investigation of \(p\)-Stam inequalities conjectured the Hermite coupling spectrum from numerical evidence~\cite[Conjecture~4.1]{Hashemi26PStam}.
\Cref{thm:exact-spectrum} proves this conjecture and extends it to
unequal Hermite scales.  We identify the coupling modes with the classical
root spectrum~\cite{Sasaki15} and finite free cumulant evolution~\cite{AP18}, and obtain a local estimate in Euclidean root distance with the exact dyadic rate. 

After the preliminaries, \Cref{sec:matrix-defect} proves the matrix
defect and the split classification.  Their applications to equality
and strict transport occupy
\Cref{sec:information-equality,sec:strict-transport}.
The local analysis begins in \Cref{sec:collision-local} with the
tangent convolution and information formulas,
\Cref{sec:collision-geometry} obtains the real singularity arrangement,
covering monodromy, and discriminant extremizers, and ends with the
entropy criterion for Stam equality.  The appendices contain the
Hermite spectrum and dynamics, followed by the geometric interpretations.

%\section*{Acknowledgements}The first author thanks Jorge Garza-Vargas for questions and comments thatmotivated the equality and spectral problems, and Gergely Berczi and Johannes Schmitt for discussions of the computational methodology.

\section{Finite free preliminaries}\label{sec:preliminaries}

We work with monic real-rooted polynomials of a fixed degree \(n\).
Roots are listed increasingly whenever they are used as coordinates.  If
\begin{equation}\label{eq:coefficients}
 p(x)=\sum_{k=0}^n(-1)^ka_kx^{n-k},\qquad
 q(x)=\sum_{k=0}^n(-1)^kb_kx^{n-k},
\end{equation}
where \(a_0=b_0=1\), then
\begin{equation}\label{eq:convolution-coefficients}
 (p\boxplus_nq)(x)=\sum_{k=0}^n(-1)^kc_kx^{n-k},
 \qquad
 c_k=\sum_{i+j=k}
 \frac{(n-i)!(n-j)!}{n!(n-k)!}a_ib_j.
\end{equation}
Walsh's theorem implies that \(p\boxplus_nq\) is real-rooted whenever both
inputs are real-rooted \cite{Walsh22}.  Marcus, Spielman, and Srivastava made
this convolution a basic operation of finite free probability and developed
its structural properties \cite{MSS22}.  It is commutative and associative,
commutes with translations, and adds means and variances.  We use the root variance
\begin{equation}\label{eq:variance}
 \Var(p)=\frac1n\sum_{i=1}^n(\alpha_i-\bar\alpha)^2,
 \qquad \bar\alpha=\frac1n\sum_i\alpha_i.
\end{equation}

For a root vector \(\alpha\), put
\(P_\alpha(x)=\prod_{i=1}^n(x-\alpha_i)\).  We use the increasingly ordered vector of roots whenever the inputs are real.
When discussing affine trajectories through collisions, equalities of root
families will instead mean equalities of unordered multisets.  We write
\begin{equation}\label{eq:root-map}
 \Omega(\alpha,\beta)=\operatorname{roots}(P_\alpha\boxplus_nP_\beta)
\end{equation}
for the ordered root map.  Wherever the output is simple, the
implicit-function theorem makes its roots analytic functions of the
input coefficients.  Those coefficients are polynomial in the input
root coordinates, so the full Jacobian \(J=D\Omega(\alpha,\beta)\)
exists even when an input has repeated coordinates.  Here the input
coordinates range over an open subset of \(\R^{2n}\) and only the output
roots need to be ordered to choose their analytic labels. The following permutation formula also appears in \cite[Theorem~1.1, equation~(5)]{GorinMarcus20}. We include a coefficient proof because the matching interpretation will be used in the analysis of collisions.

\begin{lemma}[Permutation formula for finite free convolution]
\label{lem:permutation-formula}
For \(\alpha,\beta\in\R^n\),
\begin{equation}\label{eq:hyperbolic-F}
 \mathcal F_n(x,\alpha,\beta)
 :=\frac1{n!}\sum_{\pi\in S_n}
 \prod_{r=1}^n(x-\alpha_r-\beta_{\pi(r)})
 =P_\alpha\boxplus_nP_\beta.
\end{equation}
\end{lemma}

\begin{proof}
The signed coefficient of \(x^{n-k}\) on the left is
\[
 \frac1{n!}\sum_{\pi\in S_n}e_k(\alpha+\pi\beta)
 =\sum_{i+j=k}e_i(\alpha)e_j(\beta)
   \frac{\binom{n-i}{j}}{\binom nj}.
\]
Since
\[
 \frac{\binom{n-i}{j}}{\binom nj}
 =\frac{(n-i)!(n-j)!}{n!(n-k)!},
\]
this is exactly the coefficient \(c_k\) in
\eqref{eq:convolution-coefficients}.
\end{proof}

Fujie's atom theorem will be used in the following precise form
\cite[Propositions~3.6 and~3.8]{Fujie26}.  If roots \(a\) and \(b\) of
the two inputs have multiplicities \(m\) and \(k\), and
\(m+k\geq n\), then \(a+b\) is an output root of multiplicity
\(m+k-n\), with multiplicity zero meaning that it is absent.  Every output
root which is not of this form is simple.  It follows at once that
convolution with a simple-rooted factor has only simple roots: a repeated
output root cannot be nontrivial, while a trivial one would require
\(m+1-n\geq2\), impossible because \(m\leq n\).  Thus either simple
input suffices for the regularity of the root map just described.

Define the separately centered subspace
\begin{equation}\label{eq:centered-subspace}
 \mathcal V=\{u\oplus v:\langle u,\one\rangle=0,
 \ \langle v,\one\rangle=0\}.
\end{equation}
For two simple inputs, the score-transport identity
\cite[Lemma~3.2]{GVS26} gives, for all \(a,b\in\R\),
\begin{equation}\label{eq:score-transport}
 J\bigl(a s(\alpha)\oplus b s(\beta)\bigr)
 =(a+b)s(\gamma),\qquad \gamma=\Omega(\alpha,\beta).
\end{equation}
The score sums to zero, so the input vector on the left belongs to
\(\mathcal V\).  Thus a norm estimate for the Jacobian becomes an information inequality
when applied to score vectors.  The next section proves such an estimate
and determines all its equality directions.

The heat polynomials in \eqref{eq:intro-heat-hermite} satisfy
\begin{equation}\label{eq:hermite-semigroup}
 H_{n,\tau}\boxplus_nH_{n,\sigma}=H_{n,\tau+\sigma},
 \qquad \Var(H_{n,\tau})=(n-1)\tau.
\end{equation}
At a root \(h_i\) of \(H_{n,\tau}\), the Hermite differential equation
gives
\begin{equation}\label{eq:hermite-score}
 s(h)_i=\frac{H_{n,\tau}''(h_i)}{2H_{n,\tau}'(h_i)}
 =\frac{h_i}{2\tau}.
\end{equation}

\section{Finite free curves, the matrix defect, and the split locus}
\label{sec:matrix-defect}

We first prove the defect identity and illustrate it in degree two.  The
classification of its zero directions then reduces to a lemma about affine
root trajectories.  The only external input to that lemma is the collision
multiplicity theorem recalled in the preceding section. We use the polynomial \(\mathcal F_n\) from \Cref{lem:permutation-formula}.  As a polynomial in \(x,\zeta,\eta\), it is homogeneous and hyperbolic in the \(x\)-direction and its \(x\)-roots are \(\Omega(\zeta,\eta)\)~\cite[Proposition~1.6]{LeakeRyder18}.

\begin{proof}[Proof of \Cref{thm:intro-matrix-defect}]
Fix \(w=u\oplus v\) and restrict \(\mathcal F_n\) to the ternary
homogeneous polynomial
\begin{equation}\label{eq:ternary-restriction}
 h(x,y,t)=\mathcal F_n(x,y\alpha+tu,y\beta+tv).
\end{equation}
It is hyperbolic in the \(x\)-direction and \(h(1,0,0)=1\).  The
Helton--Vinnikov theorem, equivalently the Lax conjecture in three
variables~\cite{HeltonVinnikov07,LewisParriloRamana05}, gives real
symmetric matrices \(A,B\) such that
\begin{equation}\label{eq:determinantal-pencil}
 h(x,y,t)=\det(xI-yA-tB).
\end{equation}
The eigenvalues of \(A\) are the distinct numbers \(\gamma_i\).  Work in an
orthonormal eigenbasis of \(A\).  First-order perturbation of a simple
eigenvalue says that the velocity of the \(i\)-th root of
\(\det(xI-A-tB)\) is \(B_{ii}\).  The root map computes the same velocity
as \((Jw)_i\), and therefore
\begin{equation}\label{eq:diagonal-is-jacobian}
 \operatorname{diag}_A(B)=Jw.
\end{equation}

We have identified the diagonal part of \(B\). To determine the energy in its off-diagonal entries, we compute its full Frobenius norm from mean and variance additivity. For every real \(t\),
\begin{equation}\label{eq:second-moment-identity}
 \sum_i\gamma_i(t)^2
 =\lVert\alpha+tu\rVert_2^2+\lVert\beta+tv\rVert_2^2
   +2n\,\overline{\alpha+tu}\,\overline{\beta+tv}.
\end{equation}
The left side is \(\operatorname{tr}(A+tB)^2\).  Taking one half of the
second derivative at zero yields
\begin{equation}\label{eq:pencil-frobenius-norm}
 \lVert B\rVert_{\mathrm F}^2
 =\lVert u\rVert_2^2+\lVert v\rVert_2^2+2n\bar u\bar v.
\end{equation}
The diagonal and off-diagonal parts of a symmetric matrix are orthogonal.
Together with \eqref{eq:diagonal-is-jacobian}, this gives
\begin{align}\label{eq:sum-of-squares}
 \lVert\operatorname{off}_A(B)\rVert_{\mathrm F}^2
 &=\lVert B\rVert_{\mathrm F}^2
   -\lVert\operatorname{diag}_A(B)\rVert_2^2\notag\\
 &=\lVert w\rVert_2^2-\lVert Jw\rVert_2^2+2n\bar u\bar v.
\end{align}
This proves \eqref{eq:intro-matrix-defect} for arbitrary \(w\), not merely
for separately centered directions.  Its left side depends only on the root
map, so the off-diagonal energy is independent of the chosen normalized
definite representation.

If the defect vanishes, then \(A\) and \(B\) commute.  They can be
simultaneously diagonalized, and
\begin{equation}\label{eq:split-pencil}
 h(x,y,t)=\prod_{i=1}^n(x-y\gamma_i-tB_{ii}).
\end{equation}
Conversely, because \(h\) is monic of degree \(n\) in \(x\), a complete
real linear factorization has the form
\(h(x,y,t)=\prod_i(x-y\gamma_i-tz_i)\), after matching the distinct roots
at \(t=0\).  Hence \(Jw=z\).  Comparing the coefficient of \(t^2\) in
\(\operatorname{tr}(A+tB)^2=\sum_i(\gamma_i+tz_i)^2\) gives
\(\lVert B\rVert_{\mathrm F}^2=\lVert z\rVert_2^2\).  By
\eqref{eq:diagonal-is-jacobian}, the diagonal part of \(B\) already has
this norm, so its off-diagonal part vanishes.  This proves the final
assertion.
\end{proof}

In degree two the pencil can be written down explicitly.  The calculation
also shows why the equality condition compares the relative velocities of
the two inputs.

\begin{example}[The quadratic pencil]\label{ex:quadratic-pencil}
Let \(\alpha=(-a,a)\), \(\beta=(-b,b)\),
\(u=(-r,r)\), and \(v=(-s,s)\), where \(a,b>0\).  Put
\(c=(a^2+b^2)^{1/2}\), \(d=ar+bs\), and \(e=(as-br)/c\).  Along this
direction the convolution polynomial is
\begin{equation}\label{eq:quadratic-convolution-path}
 h(x,y,t)=x^2-(ay+rt)^2-(by+st)^2.
\end{equation}
It has the representation \(h(x,y,t)=\det(xI-yA-tB)\), where
\begin{equation}\label{eq:quadratic-pencil-matrices}
 A=\begin{pmatrix}-c&0\\0&c\end{pmatrix},\qquad
 B=\begin{pmatrix}-d/c&e\\e&d/c\end{pmatrix}.
\end{equation}
Indeed,
\(d^2/c^2+e^2=r^2+s^2\).  The output roots at \(t=0\) are \((-c,c)\),
and direct differentiation gives
\begin{equation}\label{eq:quadratic-defect}
 \lVert w\rVert_2^2-\lVert Jw\rVert_2^2
 =\frac{2(as-br)^2}{a^2+b^2}
 =2e^2
 =\lVert\operatorname{off}_A(B)\rVert_{\mathrm F}^2.
\end{equation}
The defect vanishes exactly when \(as=br\).  In that case
\((r,s)\) is proportional to \((a,b)\): both centered inputs are dilated by
the same factor, and the conic in \eqref{eq:quadratic-convolution-path}
splits into two lines.
\end{example}

\begin{corollary}[Matrix-analytic finite free Stam]
\label{cor:matrix-stam}
Let \(\alpha,\beta\in\R^n\), and write \(\gamma=\Omega(\alpha,\beta)\).
Whenever the output \(P_\alpha\boxplus_nP_\beta\) is simple,
\(\lVert Jw\rVert_2\leq\lVert w\rVert_2\) for every
\(w\in\mathcal V\).  If both inputs are simple, then for all
\(a,b\in\R\),
\begin{equation}\label{eq:weighted-score-contraction}
 (a+b)^2\lVert s(\gamma)\rVert_2^2
 \leq a^2\lVert s(\alpha)\rVert_2^2
 +b^2\lVert s(\beta)\rVert_2^2.
\end{equation}
The finite free Stam inequality \eqref{eq:intro-stam} holds for all
monic real-rooted inputs of degree \(n\), with reciprocal Fisher
information defined to be zero at repeated-root polynomials.
\end{corollary}

\begin{proof}
The first assertion is immediate from
\eqref{eq:intro-matrix-defect}.  Since each score has coordinate sum zero,
apply it to \(w=a s(\alpha)\oplus b s(\beta)\) and use
\eqref{eq:score-transport}.  This proves
\eqref{eq:weighted-score-contraction}.  For simple inputs, set \(a=\lVert s(\alpha)\rVert_2^{-2}\) and \(b=\lVert s(\beta)\rVert_2^{-2}\), and the normalizing constants in
\eqref{eq:intro-fisher} cancel, giving \eqref{eq:intro-stam}.

Suppose next that \(\alpha\) is simple and \(\beta\) has repeated
coordinates.  Approximate \(\beta\) by simple real root vectors.
Fujie's theorem makes the limiting output simple, so its score and
the full Jacobian converge.  Passing to the limit in
\eqref{eq:score-transport} with \(a=1\), \(b=0\) gives
\(J(s(\alpha)\oplus0)=s(\gamma)\).  The matrix contraction yields
\(\Phi_n(P_\alpha\boxplus_nP_\beta)\leq\Phi_n(P_\alpha)\),
which is Stam because \(1/\Phi_n(P_\beta)=0\).
Commutativity treats the other mixed case.  If neither input is simple,
both terms on the right of Stam vanish and the inequality follows
from non-negativity of reciprocal Fisher information.
\end{proof}

This gives the directional matrix proof suggested in
\cite[\S5, footnote~2]{GVS26}.  It attaches a definite pencil to every
directional plane curve and extracts from that pencil a scalar independent of the chosen representation.  That directional construction is exactly what the equality problem needs.

To understand when the sum of squares vanishes, we turn from matrices to the incidence pattern of the resulting root lines.  Use the atom alternative stated after
\Cref{lem:permutation-formula}.  An output root of multiplicity
\(r\geq2\) must be trivial.  Thus it is the sum of input roots whose
multiplicities \(m,k\) satisfy, in particular,
\begin{equation}\label{eq:fujie-multiplicity}
 m+k\geq n+r.
\end{equation}
Fujie's exact statement assigns the corresponding trivial root multiplicity \(m+k-n\). The weak form \eqref{eq:fujie-multiplicity} is all the line count needs. The incidence count applies before we assume that the output motions are affine. We record this stronger form because it will also control the real singularities of non-split curves.

\begin{lemma}[Uniqueness of the real collision parameter]
\label{lem:unique-real-collision-fiber}
Let \(\alpha,\beta\in\R^n\) both have simple coordinates.
For every \(w\in\R^{2n}\), there is at most one
\([y:t]\in\mathbf P^1(\R)\) for which \(h_w(x,y,t)\), viewed as a
polynomial in \(x\), has a repeated root.
\end{lemma}

\begin{proof}
Suppose there were collisions at two distinct parameters.  Fujie's
multiplicity theorem~\cite{Fujie26} supplies first-input blocks \(I_1,I_2\) and
second-input blocks \(J_1,J_2\), of sizes \(m_1,m_2\) and
\(k_1,k_2\), such that \(m_\ell+k_\ell\geq n+2\) for
\(\ell=1,2\).

If two distinct labels belonged to both \(I_1\) and \(I_2\), the
nonzero linear form \(y(\alpha_i-\alpha_j)+t(u_i-u_j)\) would vanish at two distinct
projective parameters.  This is impossible. it is nonzero because \(\alpha_i\ne\alpha_j\).  Hence \(|I_1\cap I_2|\leq1\) and \(m_1+m_2\leq n+1\).  The same argument for the second input gives
\(k_1+k_2\leq n+1\).  Adding these bounds yields \begin{equation}\label{eq:collision-contradiction}
 2n+4\leq(m_1+k_1)+(m_2+k_2)\leq2n+2,
\end{equation}
a contradiction.  The argument takes place on the projective parameter
line and includes a possible collision at infinity.
\end{proof}

\begin{lemma}[Rigidity of an affine convolution trajectory]
\label{lem:trajectory-rigidity}
Let \(\alpha,\beta,\gamma\) be simple root vectors and let
\(u,v,z\in\R^n\).  Suppose
\begin{equation}\label{eq:abstract-affine-trajectory}
 \Omega(\alpha+tu,\beta+tv)=\gamma+tz\qquad(t\in\R)
\end{equation}
as unordered root multisets.  If the coordinates of \(z\) are not all
equal, then there is \(t_0\ne0\) at which each of the three affine root
families collapses to a single point.
\end{lemma}

\begin{proof}
Two output lines intersect because their slopes are not all equal.
\Cref{lem:unique-real-collision-fiber} shows that every output
meeting occurs at one time \(t_0\), which is nonzero because the
output is simple at \(t=0\).

All output lines intersect at the same value at \(t_0\).  Otherwise take two
lines in one collision cluster and a third line at another value.  The third line must have the slope of each of the first two and if one slope differed, that pair would intersect at a time distinct from \(t_0\).  The first two slopes are distinct, a contradiction.  The output multiplicity at \(t_0\) is therefore \(n\), so its root variance is zero. Variance additivity gives
\[
 0=\Var(P_{\alpha+t_0u}\boxplus_nP_{\beta+t_0v})
 =\Var(P_{\alpha+t_0u})+\Var(P_{\beta+t_0v}).
\]
Both summands are nonnegative.  They vanish separately, and a real root
vector has zero variance precisely when all its coordinates agree. Both
input families therefore collapse at \(t_0\).
\end{proof}

\begin{figure}[!htbp]
\centering
\begin{tikzpicture}[x=1.15cm,y=0.72cm]
  \draw[->] (-0.2,0) -- (5.2,0) node[right] {time \(t\)};
  \draw[->] (0,-0.2) -- (0,4.8) node[above] {root position};
  \draw[densely dashed] (3,0) -- (3,4.4);
  \node[below] at (3,0) {\(t_0\)};
  \draw[thick] (0,0.7) -- (4.6,3.77);
  \draw[thick] (0,1.8) -- (4.6,3.18);
  \draw[thick] (0,3.8) -- (4.6,2.11);
  \fill (3,2.7) circle (2pt);
  \node[above] at (3.25,4.15) {all outputs meet};
  \draw[->,thin] (3.25,4.08) -- (3.04,2.83);
  \fill (0,0.7) circle (1.5pt);
  \fill (0,1.8) circle (1.5pt);
  \fill (0,3.8) circle (1.5pt);
\end{tikzpicture}
\caption{The affine output-root lines in
\Cref{lem:trajectory-rigidity}.  The multiplicity count rules out two
different meeting times.  Once all output lines intersect at \(t_0\), variance
additivity forces both input root families to collapse there as well.}
\label{fig:affine-root-lines}
\end{figure}
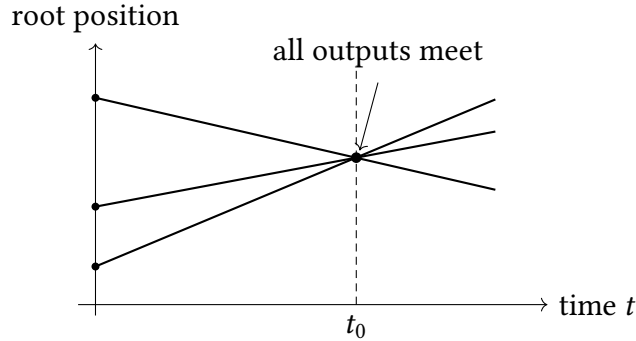

\begin{proof}[Proof of \Cref{thm:intro-split-locus}]
Write \(a=\bar u\), \(b=\bar v\),
\(u^\circ=u-a\one\), and \(v^\circ=v-b\one\).  Translation covariance of
finite free convolution gives
\begin{equation}\label{eq:translation-reduction-split}
 h_w(x,y,t)=h_{u^\circ\oplus v^\circ}
              (x-t(a+b),y,t).
\end{equation}
This invertible linear change preserves complete splitting, so it is enough
to classify \(w\in\mathcal V\).

Suppose first that \(h_w\) splits.  By
\Cref{thm:intro-matrix-defect}, the Jacobian defect vanishes, and the
factorization can be written
\begin{equation}\label{eq:global-affine-roots}
 h_w(x,y,t)=\prod_{i=1}^n(x-y\gamma_i-tz_i).
\end{equation}
Hence \(\Omega(\alpha+tu,\beta+tv)=\gamma+tz\) as an unordered multiset
for every real \(t\).  If all \(z_i\) are equal, their sum is zero because
\(w\in\mathcal V\), so \(z=0\).  The zero-defect identity then gives
\(\lVert w\rVert_2^2=\lVert Jw\rVert_2^2=0\).  Otherwise
\Cref{lem:trajectory-rigidity} gives a common collapse time \(t_0\ne0\).
The input means are constant, and therefore
\begin{equation}\label{eq:radial-split-direction}
 u=-t_0^{-1}\alpha^\circ,
 \qquad v=-t_0^{-1}\beta^\circ.
\end{equation}
Together with \eqref{eq:translation-reduction-split}, this proves the
necessity in the classification.  Conversely, if
\eqref{eq:intro-split-directions} holds, translation and dilation covariance
give
\begin{equation}\label{eq:affine-output-radial-input}
 \Omega(\alpha+tu,\beta+tv)
 =\bar\gamma\one+t(a+b)\one+(1+ct)\gamma^\circ.
\end{equation}
This identity holds as an unordered root multiset when \(1+ct<0\),
as well as before the collision.  Homogenizing it gives the explicit factorization
\begin{equation}\label{eq:radial-split-factorization}
 h_w(x,y,t)=\prod_{i=1}^n
 \bigl(x-y\gamma_i-t(a+b+c\gamma_i^\circ)\bigr),
\end{equation}
so \(h_w\) splits into \(n\) real lines.
\end{proof}

\begin{proof}[Proof of \Cref{cor:intro-jacobian-rigidity}]
Differentiate
\eqref{eq:affine-output-radial-input} with \(a=b=0\) and \(c=1\):
\begin{equation}\label{eq:radial-jacobian-vector}
 J(\alpha^\circ\oplus\beta^\circ)=\gamma^\circ.
\end{equation}
Variance additivity says
\begin{equation}\label{eq:radial-norm-additivity}
 \lVert\gamma^\circ\rVert_2^2
 =\lVert\alpha^\circ\rVert_2^2
  +\lVert\beta^\circ\rVert_2^2.
\end{equation}
Thus \(1\) is a singular value of \(J|_{\mathcal V}\).  If any other
\(w\in\mathcal V\) has \(\lVert Jw\rVert_2=\lVert w\rVert_2\), then
\Cref{thm:intro-matrix-defect} makes \(h_w\) split, and \Cref{thm:intro-split-locus}
puts \(w\) in
\(\operatorname{span}\{\alpha^\circ\oplus\beta^\circ\}\).  This proves
simplicity.
\end{proof}

The classification applies to every tangent direction.  We now use it for
weighted score vectors, where simultaneous dilation becomes the radial
condition characterizing Hermite roots.

\section{Equality in finite free information inequalities}
\label{sec:information-equality}

The split-locus theorem was proved for arbitrary tangent directions.  Score vectors form one particular family of directions, and the content of this section is that no further geometry is needed for them. Two configurations are extremal exactly when their scores are radial, and by \eqref{eq:hermite-score} radial scores mean Hermite roots. The remaining questions of which weights, which normalizations, and how to cross from Fisher information to entropy do fix the relation between the two Hermite scales.

\subsection{Weighted scores and Hermite roots}

The score of a simple root configuration sums to zero, so a weighted pair
of scores lies in the separately centered tangent space.  Equality in its
norm contraction is therefore covered by the split-locus theorem.  The
weights must be kept in the calculation: they determine how the scales of
the two Hermite inputs are related.

\begin{proposition}[Equality in weighted score contraction]
\label{prop:weighted-score-equality}
Let \(a,b>0\), and let \(\alpha,\beta\) have simple coordinates.  Equality
holds in \eqref{eq:weighted-score-contraction} if and only if, up to
independent translations,
\begin{equation}\label{eq:weighted-hermite-parameters}
 P_\alpha=H_{n,\tau_\alpha},\qquad
 P_\beta=H_{n,\tau_\beta},\qquad
 \frac{\tau_\alpha}{a}=\frac{\tau_\beta}{b}.
\end{equation}
\end{proposition}

\begin{proof}
Assume equality.  With
\(w=a s(\alpha)\oplus b s(\beta)\),
\eqref{eq:score-transport} shows that equality in
\eqref{eq:weighted-score-contraction} is
\(\lVert Jw\rVert_2=\lVert w\rVert_2\).
Hence \Cref{cor:intro-jacobian-rigidity} gives a scalar \(\rho\) such that
\begin{equation}\label{eq:radial-scores}
 a s(\alpha)=\rho\alpha^\circ,
 \qquad b s(\beta)=\rho\beta^\circ.
\end{equation}
For every simple root vector \(r\), pairwise symmetrization gives
\begin{equation}\label{eq:radial-pairing}
 \langle r^\circ,s(r)\rangle
 =\sum_{i<j}\frac{r_i-r_j}{r_i-r_j}
 =\binom n2.
\end{equation}
Taking the inner product of \eqref{eq:radial-scores} with either centered
input and using \eqref{eq:radial-pairing} shows \(\rho>0\).  Thus
\(c_\alpha=\rho/a\) and \(c_\beta=\rho/b\) are positive radial-score
constants.

Translate \(\alpha\) to mean zero and set
\(P(x)=\prod_i(x-\alpha_i)\).  At every root,
\begin{equation}\label{eq:score-log-derivative}
 \frac{P''(\alpha_i)}{2P'(\alpha_i)}=c_\alpha\alpha_i.
\end{equation}
The polynomial \(P''-2c_\alpha xP'+2c_\alpha nP\) has degree at most
\(n-1\) and vanishes at all \(n\) roots, so it is zero.  After
\(y=\sqrt{2c_\alpha}\,x\), this is the probabilists' Hermite equation.
Thus \(P=H_{n,\tau_\alpha}\) with
\(\tau_\alpha=(2c_\alpha)^{-1}=a/(2\rho)\), and similarly
\(\tau_\beta=b/(2\rho)\).  This proves \eqref{eq:weighted-hermite-parameters}.

Conversely, suppose \(\tau_\alpha/a=\tau_\beta/b\), and put
\(\rho=a/(2\tau_\alpha)=b/(2\tau_\beta)\).  By
\eqref{eq:hermite-score},
\(a s(\alpha)\oplus b s(\beta)
=\rho(\alpha^\circ\oplus\beta^\circ)\).
Equations \eqref{eq:radial-jacobian-vector} and
\eqref{eq:radial-norm-additivity} show that \(J\) preserves the norm of this
vector.  Together with \eqref{eq:score-transport}, this is equality in
\eqref{eq:weighted-score-contraction}.
\end{proof}

\begin{proof}[Proof of Stam equality in \Cref{thm:information-equality}]
Put
\(I_\alpha=\lVert s(\alpha)\rVert_2^2\),
\(I_\beta=\lVert s(\beta)\rVert_2^2\), and choose
\(a=I_\alpha^{-1}\), \(b=I_\beta^{-1}\) in
\eqref{eq:weighted-score-contraction}.  Equality in Stam is exactly
equality in that contraction.  \Cref{prop:weighted-score-equality} makes
both inputs independently scaled and translated Hermite polynomials.
Conversely, if their heat parameters are \(\tau,\sigma>0\), then
\(H_{n,\tau}\boxplus_nH_{n,\sigma}=H_{n,\tau+\sigma}\), and
\eqref{eq:hermite-score} shows that reciprocal Fisher informations add.
Translations do not change the score.
\end{proof}

For the entropy argument we will need the following weighted form.  If
\(0<\lambda<1\), applying score contraction to the scaled inputs gives
\begin{equation}\label{eq:intro-weighted-fisher}
 \Phi_n\!\left(\sqrt\lambda_*f\boxplus_n
                 \sqrt{1-\lambda}_*g\right)
 \leq\lambda\Phi_n(f)+(1-\lambda)\Phi_n(g).
\end{equation}
The unscaled inputs must have the same variance for equality to hold.

\begin{corollary}[Weighted Fisher equality]
\label{cor:weighted-fisher-equality}
Let \(0<\lambda<1\).  Equality in \eqref{eq:intro-weighted-fisher} holds
for simple-rooted \(p,q\) if and only if they are independently translated
Hermite polynomials with the same variance.
\end{corollary}

\begin{proof}
Apply \eqref{eq:weighted-score-contraction} to the scaled root vectors
\(\sqrt\lambda\,\alpha\) and \(\sqrt{1-\lambda}\,\beta\), with
\(a=\lambda\), \(b=1-\lambda\).  Since score is homogeneous of degree
\(-1\), this is exactly \eqref{eq:intro-weighted-fisher}.  In the original
unscaled variables, \Cref{prop:weighted-score-equality} says that the two
heat parameters agree.  The converse follows from
\eqref{eq:hermite-semigroup}.
\end{proof}

\subsection{Entropy concavity and entropy power}\label{sec:entropy-power}

We use a heat-flow interpolation to pass from the weighted Fisher
inequality to entropy.  Keeping track of its derivative will also identify
equality, including the endpoint at which the original polynomials are
recovered. The endpoint is where the interpolation degenerates, and it is the only place at which equality can be lost without the derivative noticing.

The entropy and entropy power from \eqref{eq:intro-entropy} scale according to
\begin{equation}\label{eq:entropy-scaling}
 \chi_n[c_*p]=\chi_n[p]+\log c,
 \qquad \cN_n(c_*p)=c^2\cN_n(p)
 \qquad(c>0).
\end{equation}

Let \(\check H_n=H_{n,1/(n-1)}\), so that
\(\Var(\check H_n)=1\) and \(\Phi_n(\check H_n)=1\).  The finite free de
Bruijn identity of~\cite{GVS26} is
\begin{equation}\label{eq:debruijn}
 \frac{d}{dt}\chi_n[p\boxplus_n\sqrt t_*\check H_n]
 =\frac12\Phi_n(p\boxplus_n\sqrt t_*\check H_n).
\end{equation}
Together with the weighted Fisher inequality, it yields the entropy
concavity statement
\begin{equation}\label{eq:entropy-concavity}
 \chi_n[\sqrt\lambda_*p\boxplus_n\sqrt{1-\lambda}_*q]
 \geq\lambda\chi_n[p]+(1-\lambda)\chi_n[q],
 \qquad 0\leq\lambda\leq1.
\end{equation}
The inequality is proved in~\cite{GVS26}, following the classical
normal-perturbation interpolation of Dembo, Cover, and
Thomas~\cite{DCT91}.  We give the interpolation below because its
derivative identifies equality.

\begin{theorem}[Equality in entropy concavity]
\label{thm:entropy-concavity-equality}
Let \(0<\lambda<1\), and let \(p,q\) have simple real roots.  Equality
holds in \eqref{eq:entropy-concavity} if and only if \(p\) and \(q\) are
independently translated Hermite polynomials with the same variance.
\end{theorem}

\begin{proof}
We interpolate from a common unit-variance Hermite polynomial at \(t=0\) to the prescribed inputs at \(t=1\). The entropy defect starts at zero, and its derivative will be controlled by the weighted Fisher inequality. For \(0\leq t\leq1\), define
\begin{equation}\label{eq:entropy-interpolation}
 p_t=\sqrt t_*p\boxplus_n\sqrt{1-t}_*\check H_n,
 \qquad
 q_t=\sqrt t_*q\boxplus_n\sqrt{1-t}_*\check H_n,
\end{equation}
and
\begin{equation}\label{eq:r-interpolation}
 r_t=\sqrt\lambda_*p_t\boxplus_n\sqrt{1-\lambda}_*q_t.
\end{equation}
The Hermite semigroup law implies
\(r_t=\sqrt t_*r_1\boxplus_n\sqrt{1-t}_*\check H_n\).  Set
\begin{equation}\label{eq:entropy-defect-path}
 F(t)=\chi_n[r_t]-\lambda\chi_n[p_t]-(1-\lambda)\chi_n[q_t].
\end{equation}
At \(t=0\), all three polynomials equal \(\check H_n\), so \(F(0)=0\).
For \(0\leq t<1\), each of \(p_t,q_t,r_t\) is a convolution with a
positive-scale Hermite polynomial, hence is simple by the consequence of
Fujie's atom theorem recorded after \Cref{lem:permutation-formula}.  At
\(t=1\), the same argument applies because the relevant inputs are simple.
Coefficient continuity and simplicity make
the ordered roots, scores, and \(F\) continuous on \([0,1]\), with \(F\)
differentiable on \((0,1)\).

For \(t>0\), factor \(\sqrt t\) from \eqref{eq:entropy-interpolation} and
apply \eqref{eq:entropy-scaling} and \eqref{eq:debruijn} with
\(\epsilon_t=(1-t)/t\).  Since
\(\epsilon_t'=-t^{-2}\) and Fisher information has scaling degree \(-2\),
one obtains
\begin{equation}\label{eq:entropy-path-derivative}
 F'(t)=\frac{1}{2t}\left(
 \lambda\Phi_n(p_t)+(1-\lambda)\Phi_n(q_t)-\Phi_n(r_t)
 \right)\geq0.
\end{equation}
The last inequality is \eqref{eq:intro-weighted-fisher}.  Thus
\eqref{eq:entropy-concavity} is \(F(1)\geq F(0)\).

If equality holds at \(t=1\), then the continuous nonnegative derivative
in \eqref{eq:entropy-path-derivative} vanishes for \(0<t<1\).  Letting
\(t\uparrow1\), and using continuity of the score on the simple-rooted
locus, gives equality in \eqref{eq:intro-weighted-fisher} for \(p,q\).
\Cref{cor:weighted-fisher-equality} gives the asserted Hermite form and
the common variance.  Conversely, equal-variance Hermite inputs make every
polynomial in \eqref{eq:entropy-interpolation} Hermite with the same
variance, so \(F(t)=0\) identically.
\end{proof}

\begin{theorem}[Finite free entropy-power equality]
\label{thm:epi-equality}
Let \(p,q\) be monic degree-\(n\) polynomials with simple real roots.  Then
\begin{equation}\label{eq:epi}
 \cN_n(p\boxplus_nq)\geq\cN_n(p)+\cN_n(q),
\end{equation}
and equality holds if and only if \(p\) and \(q\) are independently
translated and positively scaled Hermite polynomials.
\end{theorem}

\begin{proof}
The inequality is the finite free entropy power inequality of~\cite{GVS26}.
We prove equality by reducing it to the entropy-concavity case. Normalize the two inputs separately to have entropy power one.  Put
\(A=\cN_n(p)>0\), \(B=\cN_n(q)>0\),
\begin{equation}\label{eq:epi-normalization}
 P=A^{-1/2}_*p,
 \qquad Q=B^{-1/2}_*q,
 \qquad \lambda=\frac{A}{A+B}.
\end{equation}
Then \(\cN_n(P)=\cN_n(Q)=1\), and dilation covariance of finite
free convolution gives
\begin{equation}\label{eq:epi-to-concavity}
 (A+B)^{-1/2}_*(p\boxplus_nq)
 =\sqrt\lambda_*P\boxplus_n\sqrt{1-\lambda}_*Q.
\end{equation}
Consequently, equality in \eqref{eq:epi} is equivalent to equality in
\eqref{eq:entropy-concavity} for \(P,Q\).  By
\Cref{thm:entropy-concavity-equality}, both are equal-variance translated
Hermite polynomials.  Undoing the two independent normalizations in
\eqref{eq:epi-normalization} gives independently scaled Hermite inputs. Conversely, for translated copies of \(H_{n,\tau}\) and \(H_{n,\sigma}\), the semigroup law gives \(H_{n,\tau+\sigma}\).  Since
entropy power has scaling degree two, \(\cN_n(H_{n,\tau})=\tau\cN_n(H_{n,1})\), and equality follows.
\end{proof}

This proves the entropy-power assertion of \Cref{thm:information-equality}. The common-variance condition belongs to the weighted interpolation and the independent normalization of the two entropy powers is what permits arbitrary positive Hermite scales in the final equality statement.

\section{Strict contraction under finite free convolution}
\label{sec:strict-transport}

While convolution with a fixed real-rooted factor preserves majorization of root vectors~\cite{BorceaBranden10,LeakeRyder18}, we ask the question whether holding one input fixed, does every output root respond when a single root of the other input moves? The split-locus theorem identifies the separately centered motions that preserve the joint Euclidean norm when both inputs have simple roots. For the motion of one input alone, Fujie's multiplicity theorem gives a stronger conclusion that also allows repeated roots in the fixed factor. The essential condition on that factor is that its roots are not all equal. We call a polynomial of the form \((x-b)^n\) a \emph{point mass}. 
The dichotomy is complete and there is nothing between its two halves, maening that a point mass translates every root rigidly and mixes nothing, while every other real-rooted factor of degree \(n\) mixes every input velocity into every output velocity, with a strictly positive weight on each. No hypothesis on the multiplicities of the fixed factor is needed for this, which is what distinguishes the argument from the one in \Cref{sec:matrix-defect}, and it is also why the determinantal representation plays no role below. We move one input root at a time and show that every output root must respond. This gives strict contraction of the score in every \(\ell^p\) norm with \(p>1\), as well as strict transport of centered root configurations.

Differentiating with respect to one input root removes one linear factor. We therefore need to understand how convolution with the fixed factor acts on polynomials of degree \(n-1\). We first record how the frozen convolution operator acts on polynomials
of smaller degree.  If
\(q(x)=\sum_{k=0}^n(-1)^k b_kx^{n-k}\), with \(b_0=1\), define
\begin{equation}\label{eq:frozen-convolution-operator}
 T_q=\sum_{k=0}^n(-1)^k b_k\frac{(n-k)!}{n!}\partial_x^k.
\end{equation}
The coefficient formula for convolution gives \(T_qf=f\boxplus_nq\)
for every monic degree-\(n\) polynomial \(f\).  The operator itself is
defined on every polynomial.  In particular, if \(u\) is monic of degree
\(n-1\), then
\begin{equation}\label{eq:lower-degree-frozen-convolution}
 T_qu=u\boxplus_{n-1}\frac{q'}n.
\end{equation}
Indeed, the coefficient indexed by \(k\) in \(q'/n\) is
\(b_k(n-k)/n\), and
\(\frac{(n-k-1)!}{(n-1)!}\frac{n-k}{n}
=\frac{(n-k)!}{n!}\) for \(0\leq k\leq n-1\).
When \(n=2\), the right side of \eqref{eq:lower-degree-frozen-convolution} is a degree-one convolution and as usual, \((x-a)\boxplus_1(x-b)=x-a-b\).

\begin{theorem}[Positivity of the one-sided root Jacobian]
\label{thm:positive-one-sided-jacobian}
Let \(n\geq2\), let \(\alpha_1<\cdots<\alpha_n\), and let \(q=P_\beta\)
be a monic real-rooted polynomial of degree \(n\) that is not a point
mass.  The roots \(\gamma=\Omega(\alpha,\beta)\) are simple.  The matrix
\(E=D_1\Omega(\alpha,\beta)\) is doubly stochastic, and every entry of
\(E\) is strictly positive.
\end{theorem}

\begin{proof}
The simplicity of \(P_\alpha\) implies that of the output by Fujie's
atom theorem.  The coefficients of convolution depend polynomially on
the input root coordinates.  The implicit-function theorem therefore makes the ordered output roots, and their first derivatives, analytic near \((\alpha,\beta)\), even if \(\beta\) has repeated coordinates. Approximate \(\beta\) by simple real root vectors. The corresponding one-sided Jacobians are doubly stochastic by \cite[Lemma~3.4]{GVS26} and passage to the limit proves the same assertion for \(E\). In particular, its entries are nonnegative.

Because the output polynomial depends linearly on the moving root’s position, a zero output velocity would force that output root to remain fixed throughout the motion. We show that such a fixed root would become multiple at some position. Fix an input index \(j\), and put \(p_j(x)=\prod_{k\ne j}(x-\alpha_k)\). Use \(t\) for the position of the \(j\)-th root, so the original input is recovered at \(t=\alpha_j\).
Moving only this input root gives \(f_t(x)=(x-t)p_j(x)\).  By linearity of \(T_q\), write 

\begin{equation}\label{eq:single-root-motion}
 r_t=T_qf_t=a_j-tb_j,
 \qquad a_j=T_q(xp_j),\qquad b_j=T_qp_j.
\end{equation}
Equation~\eqref{eq:lower-degree-frozen-convolution} gives \(b_j=p_j\boxplus_{n-1}(q'/n)\).  Here \(p_j\) is simple-rooted and \(q'/n\) is monic and real-rooted by Rolle's theorem.  Thus \(b_j\) is a simple-rooted polynomial of degree \(n-1\).  For degree one this is trivial and in higher degree it follows again from Fujie's atom theorem.

Every \(r_t\), for \(t\in\R\), also has simple roots.  Indeed, \(f_t\) has root multiplicities at most two, since the roots of \(p_j\) are distinct.  The multiplicities of \(q\) are at most \(n-1\), since \(q\) is not a point mass.  A repeated output root would require input multiplicities whose sum is at least \(n+2\), whereas here their sum is at most \(n+1\).

At \(t=\alpha_j\), differentiating the equation for the output root
\(\gamma_i\) gives
\(E_{ij}=b_j(\gamma_i)/r_{\alpha_j}'(\gamma_i)\).
If \(E_{ij}=0\), then \(b_j(\gamma_i)=0\), and
\eqref{eq:single-root-motion} also gives \(a_j(\gamma_i)=0\).
Consequently, \(\gamma_i\) is a root of every \(r_t\).
Since this root of \(b_j\) is simple, choosing
\(t=a_j'(\gamma_i)/b_j'(\gamma_i)\) makes
\(r_t'(\gamma_i)=0\).  This contradicts the simplicity of \(r_t\)
proved above.  Thus \(E_{ij}\ne0\); its non-negativity proves the claim.
\end{proof}

Strict positivity means that each output velocity is an average in which every input velocity has positive weight. For a strictly convex function, Jensen’s inequality is therefore strict on every non-constant input vector. Applying this observation to the score yields the following consequence.

\begin{theorem}[Strict score monotonicity]
\label{thm:strict-score-monotonicity}
Let \(f=P_\alpha\) be monic of degree \(n\geq2\) with simple real roots,
and let \(q=P_\beta\) be monic and real-rooted of the same degree.
Write \(\gamma=\Omega(\alpha,\beta)\).  For every \(p>1\),
\(\|s(\gamma)\|_p\leq\|s(\alpha)\|_p\), with equality if and only if
\(q\) is a point mass.
\end{theorem}

\begin{proof}
Suppose first that \(q\) is not a point mass, and use the matrix \(E\)
from \Cref{thm:positive-one-sided-jacobian}.  The one-sided score-transport
identity \(s(\gamma)=Es(\alpha)\), including repeated coordinates in
\(\beta\), was established in the proof of \Cref{cor:matrix-stam}.

The vector \(s(\alpha)\) is non-constant.  Its coordinates sum to zero,
while pairwise symmetrization gives
\(\langle\alpha^\circ,s(\alpha)\rangle=\binom n2>0\).
Strict convexity of \(x\mapsto|x|^p\), positivity of every entry of
\(E\), and its row and column sums therefore give
\begin{equation}\label{eq:strict-score-jensen}
 \sum_i\left|\sum_j E_{ij}s(\alpha)_j\right|^p
 <\sum_{i,j}E_{ij}|s(\alpha)_j|^p
 =\sum_j|s(\alpha)_j|^p.
\end{equation}
This proves the strict inequality.  If \(q=(x-b)^n\), convolution translates every root of \(f\) by \(b\), all root gaps, and hence all score coordinates, are unchanged.
\end{proof}

The same averaging argument contracts every non-constant velocity vector. Integrating this contraction along a path of input configurations gives a comparison between its endpoints. For increasing root vectors, write
\(d_2(r,\widetilde r)=\|r^\circ-\widetilde r^\circ\|_2\), where
\(r^\circ=r-\bar r\one\).  This distance removes translations.
The centered empirical measure
\(\mu_r^\circ=n^{-1}\sum_i\delta_{r_i-\bar r}\) satisfies
\(d_2(r,\widetilde r)=\sqrt n\,W_2(\mu_r^\circ,\mu_{\widetilde r}^\circ)\),
because increasing matching is optimal on the real line.

\begin{theorem}[Strict centered root transport]
\label{thm:strict-root-transport}
Let \(q\) be a monic real-rooted polynomial of degree \(n\geq2\).
Convolution by \(q\) strictly contracts \(d_2\) between every pair of
simple root configurations that are distinct modulo translation if and
only if \(q\) is not a point mass.  Equivalently, this is the necessary
and sufficient condition for strict contraction of \(W_2\) between
distinct centered empirical measures of simple degree-\(n\) polynomials.
\end{theorem}

\begin{proof}
Suppose that \(q=P_\beta\) is not a point mass.  Let
\(\alpha,\widetilde\alpha\) be increasing simple root vectors with
\(v=\widetilde\alpha^\circ-\alpha^\circ\ne0\), and set
\(\alpha_t=(1-t)\alpha^\circ+t\widetilde\alpha^\circ\) for
\(0\leq t\leq1\).  The ordered simple chamber is convex, so each
\(\alpha_t\) is simple.  By
\Cref{thm:positive-one-sided-jacobian}, the matrix
\(E_t=D_1\Omega(\alpha_t,\beta)\) has strictly positive entries and is
doubly stochastic.  Since the nonzero centered vector \(v\) is
non-constant, strict Jensen gives \(\|E_tv\|_2<\|v\|_2\).
Continuity on the compact segment provides \(\kappa<1\) such that
\(\|E_tv\|_2\leq\kappa\|v\|_2\) for all \(t\in[0,1]\).

The mean of \(\Omega(\alpha_t,\beta)\) is the constant \(\bar\beta\).
Thus centering this output path does not change its derivative, which
is \(E_tv\).  Translation covariance and the fundamental theorem of
calculus now yield
\begin{equation}\label{eq:strict-transport-segment}
 d_2\bigl(\Omega(\alpha,\beta),
          \Omega(\widetilde\alpha,\beta)\bigr)
 \leq\int_0^1\|E_tv\|_2\,dt
 \leq\kappa\|v\|_2
 <d_2(\alpha,\widetilde\alpha).
\end{equation}
The factor \(\kappa\) depends on the fixed polynomial and the chosen
input segment.  The empirical-measure formulation follows from the
identity between \(d_2\) and \(\sqrt n\,W_2\).
If \(q=(x-b)^n\), convolution is translation by \(b\), so every centered
root distance is preserved.  This also proves necessity.
\end{proof}

\begin{example}[A repeated-root factor that strictly mixes velocities]
\label{ex:strict-transport-repeated-factor}
Take \(q=x^{n-1}(x-b)\), with \(b\ne0\).  For \(n\geq3\) this factor
has repeated roots, but it is not a point mass.  Its frozen operator is
\(T_q=I-(b/n)\partial_x\).  If \(f=P_\alpha\) is simple-rooted and
\(\gamma\) is the output root vector, no \(\gamma_i\) equals an
\(\alpha_j\), since \(T_qf(\alpha_j)=-(b/n)f'(\alpha_j)\ne0\).
The equation \(T_qf(\gamma_i)=0\) can therefore be written as
\(\sum_j(\gamma_i-\alpha_j)^{-1}=n/b\).  Implicit differentiation gives
the explicit averaging weights
\begin{equation}\label{eq:repeated-factor-jacobian}
 E_{ij}=
 \frac{(\gamma_i-\alpha_j)^{-2}}
      {\sum_k(\gamma_i-\alpha_k)^{-2}}.
\end{equation}
Every input root contributes a positive weight to every output velocity.

For a concrete cubic, choose \(f=x^3-9x\) and \(q=x^2(x-4)\). The input roots of \(f\) are \((-3,0,3)\), and \(T_qf=x^3-4x^2-9x+12=(x-1)(x^2-3x-12)\). Its middle root is \(\gamma_2=1\). The squared inverse distances from this root to the input roots are \(1/16,1,1/4\), so the second row of
\(E\) is \((1,16,4)/21\). The input score is \((-1/2,0,1/2)\) and its average with these weights is \(1/14\), which is exactly the score of the middle output root. Thus even a factor with a double root produces a strict average of the non-constant input score.
\end{example}

\section{The convolution seen at a collision}
\label{sec:collision-local}
The output roots remained simple along the motions studied in the preceding section. We now go back to the collisions used to classify equality, and ask how the roots separate when the curve does not split.
The equality argument used collisions to determine when every output root
can follow an affine line. The collision also contains information about
curves that do not split. Fujie's theorem says how many roots meet, and we now determine the rate with which they separate. Looked at on the scale of that separation, what one finds is another finite free convolution, of degree equal to the multiplicity of the collision. This description will identify the local branches of a finite free curve and the singular terms in its entropy and Fisher information.

\subsection{The tangent convolution}

We allow repeated roots at the parameter value under consideration.  Let
\(\alpha,\beta,u,v\in\R^n\), and put
\(p_\eps=P_{\alpha+\eps u}\), \(q_\eps=P_{\beta+\eps v}\), and
\(h_\eps=p_\eps\boxplus_nq_\eps\).  Thus the input roots depend
affinely on the real parameter \(\eps\).  If \(a\) is a root of
\(p_0\) and \(b\) is a root of \(q_0\), write
\(I=\{i:\alpha_i=a\}\) and \(J=\{j:\beta_j=b\}\).  Their
cardinalities will be denoted by \(m\) and \(k\).  The polynomials
\(U(z)=\prod_{i\in I}(z-u_i)\) and
\(V(z)=\prod_{j\in J}(z-v_j)\) record the velocities within these
two groups of roots.  We include degree-one convolution in the notation,
with \((z-a)\boxplus_1(z-b)=z-a-b\).
We examine the output near \(a+b\) in the coordinate \(z=(x-a-b)/\varepsilon\). The following theorem shows that this scale captures all the output roots approaching \(a+b\), and identifies their limiting positions. It also determines the leading coefficient of the rescaled polynomial, which retains the contribution of the input roots outside the two colliding groups.

\begin{theorem}[Tangent convolution]\label{thm:tangent-convolution}
Let \(n\geq2\), let \(\alpha,\beta,u,v\in\R^n\), and suppose that
\(a,b\) have multiplicities \(m,k\) as above, with
\(r=m+k-n>0\).  Define the monic degree-\(r\) polynomial
\begin{equation}\label{eq:tangent-convolution-polynomial}
 Q(z)=
 \left(\frac{r!}{m!}U^{(m-r)}(z)\right)
 \boxplus_r
 \left(\frac{r!}{k!}V^{(k-r)}(z)\right)
\end{equation}
and the nonzero real constant
\begin{equation}\label{eq:tangent-convolution-constant}
 c_{\mathrm{tan}}=\frac{m!k!}{n!r!}
       \prod_{i\notin I}(a-\alpha_i)
       \prod_{j\notin J}(b-\beta_j).
\end{equation}
There is a polynomial \(R\in\R[z,\eps]\) such that
\begin{equation}\label{eq:tangent-convolution-expansion}
 h_\eps(a+b+\eps z)
   =\eps^r\bigl(c_{\mathrm{tan}}Q(z)+\eps R(z,\eps)\bigr).
\end{equation}
In particular, \(\eps^{-r}h_\eps(a+b+\eps z)\) converges
coefficientwise, and locally uniformly for \(z\in\C\), to \(c_{\mathrm{tan}}Q(z)\).

The \(r\) output roots approaching \(a+b\), after subtraction of
\(a+b\) and division by \(\eps\), converge as an unordered multiset
to the roots of \(Q\).  If these roots are distinct, say
\(\lambda_1<\cdots<\lambda_r\), the corresponding output roots admit
real analytic branches \(x_1(\eps),\ldots,x_r(\eps)\) through the
parameter value zero,
with
\begin{equation}\label{eq:tangent-convolution-branches}
 x_j(\eps)=a+b+\eps\lambda_j+O(\eps^2),
 \qquad 1\leq j\leq r.
\end{equation}
Within this cluster, the branches are increasingly ordered for positive
\(\eps\) and decreasingly ordered for negative \(\eps\).
\end{theorem}

\begin{proof}
Use the permutation formula of \Cref{lem:permutation-formula}.  For a
permutation \(\pi\in S_n\), let \(d_\pi\) be the number of indices
\(i\in I\) for which \(\pi(i)\in J\).  At most \(n-k\) indices
of \(I\) can be matched outside \(J\), so
\(d_\pi\geq m-(n-k)=r\).  Each such internal match contributes
the factor \(\eps(z-u_i-v_{\pi(i)})\) after the substitution
\(x=a+b+\eps z\).  Every summand is therefore divisible by
\(\eps^r\).  A summand with \(d_\pi>r\) contributes nothing to
the quotient at \(\eps=0\).

Consider a matching with exactly \(r\) internal matches, and let
\(S\subset I\), \(T\subset J\) be the matched subsets, both of
cardinality \(r\).  The remaining \(m-r=n-k\) indices of \(I\)
are matched bijectively to the complement of \(J\), and the complement
of \(I\) is matched bijectively to the remaining \(k-r=n-m\)
indices of \(J\).  There are no matches between the two complements.
At \(\eps=0\), these outside matches contribute
\[
 G=\prod_{i\notin I}(a-\alpha_i)
       \prod_{j\notin J}(b-\beta_j),
\]
independently of their assignment.  For a fixed bijection from \(S\)
to \(T\), the two outside assignments can be chosen in
\((n-k)!(n-m)!\) ways.  Summing over the \(r!\) internal
bijections gives \(r!(P_{u_S}\boxplus_rP_{v_T})(z)\).  Consequently,
the value of the quotient at zero is
\begin{equation}\label{eq:tangent-matching-count}
 \frac{G\,r!(n-k)!(n-m)!}{n!}
 \sum_{\substack{S\subset I,\ |S|=r\\ T\subset J,\ |T|=r}}
       (P_{u_S}\boxplus_rP_{v_T})(z).
\end{equation}

We have averaged over the bijections between the selected roots. It remains to average over the choices of the \(r\)-element subsets themselves. The average of the corresponding monic products is a normalized derivative of the original velocity polynomial. Differentiating a product gives
\[
 \frac{r!}{m!}U^{(m-r)}
   =\binom mr^{-1}\sum_{\substack{S\subset I\\|S|=r}}P_{u_S},
 \qquad
 \frac{r!}{k!}V^{(k-r)}
   =\binom kr^{-1}\sum_{\substack{T\subset J\\|T|=r}}P_{v_T}.
\]
The coefficient formula makes convolution bilinear in its polynomial arguments. Hence the double sum in \eqref{eq:tangent-matching-count} is \(\binom mr\binom kr Q\). Its scalar coefficient simplifies to
\(Gm!k!/(n!r!)=c_{\mathrm{tan}}\), because \(n-k=m-r\) and \(n-m=k-r\). This proves \eqref{eq:tangent-convolution-expansion} and the polynomial divisibility also proves both stated forms of convergence.

Fujie's atom theorem says that \(a+b\) has multiplicity exactly \(r\)
in \(h_0\).  The constant can also be checked directly from the
expansion: the coefficient of \(\eps^rz^r\) is
\(h_0^{(r)}(a+b)/r!=c_{\mathrm{tan}}\). The two normalized derivatives in
\eqref{eq:tangent-convolution-polynomial} are real-rooted by Rolle's
theorem, and their convolution is real-rooted by Walsh's theorem.

We have identified the limiting polynomial. We must still show that its roots describe precisely the output cluster approaching \(a+b\).  Choose a circle containing every root of \(Q\), with no root on its boundary. Local uniform convergence and Rouch\'e's theorem show that,
for sufficiently small nonzero \(\eps\), the rescaled polynomial
has exactly \(r\) roots inside this circle.  Their unscaled roots tend
to \(a+b\), so they are precisely its cluster, since \(h_0\) has
multiplicity \(r\) there.  Applying the same argument to disjoint
small circles around the distinct roots of \(Q\) proves multiset
convergence, including their multiplicities.

If \(\lambda_j\) is simple, apply the analytic implicit-function
theorem to \(c_{\mathrm{tan}}Q(z)+\eps R(z,\eps)\) at \((\lambda_j,0)\).
It gives a real analytic solution \(z_j(\eps)\) with
\(z_j(0)=\lambda_j\).  Setting
\(x_j(\eps)=a+b+\eps z_j(\eps)\) proves
\eqref{eq:tangent-convolution-branches}.  The signs of the differences
\(x_j(\eps)-x_i(\eps)=\eps(\lambda_j-\lambda_i)+O(\eps^2)\)
give the assertion about ordering.
\end{proof}

The leading configuration depends only on the velocities inside the two
colliding groups.  The other roots enter the leading polynomial through
the scalar \(c_{\mathrm{tan}}\). When the output multiplicity is \(r=1\), \(Q\) has its single root at
\(m^{-1}\sum_{i\in I}u_i+k^{-1}\sum_{j\in J}v_j\), so the
formula also describes simple atom roots.  More generally, the two
factors in \eqref{eq:tangent-convolution-polynomial} are again monic
real-rooted polynomials of the same degree.  If \(r\geq2\) and their roots are subjected to a further affine perturbation and form an atom collision, the theorem applies again in degree \(r\).  This is the recursive feature of the description where each localization remains within finite free convolution.

For affine input motions, the tangent configuration is simple whenever the output has simple roots at even one parameter value. The reason is that a repeated tangent root would force groups of input roots to share both their initial position and their velocity, producing a repeated output root for every nearby parameter.
\begin{lemma}[Simplicity of the tangent configuration]
\label{lem:affine-tangent-simplicity}
In the setting of \Cref{thm:tangent-convolution}, suppose that
\(\operatorname{Disc}_x(h_\eps)\) is not identically zero as a
polynomial in \(\eps\).  Then \(Q\) is simple-rooted.
\end{lemma}

\begin{proof}
The assertion is immediate for \(r=1\).  Suppose \(r\geq2\), and assume that \(Q\) has a root of multiplicity \(\rho\geq2\). Fujie's atom theorem supplies roots of its two degree-\(r\) factors
with multiplicities \(\mu,\nu\) such that \(\mu+\nu-r=\rho\). The factors have degree \(r\), so
\(\mu,\nu\leq r\). Since \(\rho\geq2\), both multiplicities are at least two. Iterated Rolle's theorem shows that a multiple root of a derivative of a real-rooted polynomial must be inherited from
a root of the original polynomial: each differentiation lowers an inherited multiplicity by one, and all roots newly occurring between distinct roots are simple.  Thus the two roots in question come from
values \(u_0,v_0\) occurring among the roots of \(U,V\) with respective multiplicities \(M=m-r+\mu\) and \(K=k-r+\nu\).

These input roots remain equal to \(a+\eps u_0\) and
\(b+\eps v_0\), respectively, for every \(\eps\).  For all
sufficiently small nonzero \(\eps\), no additional input root joins
either group, because the other roots have a different intercept or a
different slope.  Their multiplicities are therefore exactly \(M,K\),
and \(M+K-n=\rho\).  Fujie's theorem forces an output root of
multiplicity \(\rho\) at \(a+b+\eps(u_0+v_0)\).  The
discriminant vanishes on a punctured interval and hence identically,
contrary to the hypothesis.
\end{proof}

In particular, the tangent configuration is simple whenever the output is
simple for one parameter value.  Distinct velocities within each input
collision group are another sufficient condition, by Rolle and Fujie.
Without such a condition, the polynomial limit in
\Cref{thm:tangent-convolution} still holds.  The lemma shows that a
repeated tangent root in an affine real family forces a persistent output
collision.  One cannot then use a nonzero discriminant or finite Fisher
information on the nearby fibers.

\subsection{Discriminant and information at the boundary}

To recover the discriminant, we must keep track of every output root,
including those that do not collide.  Write
\(h_0(x)=\prod_{\ell=1}^{\kappa}(x-c_\ell)^{r_\ell}\), where
\(\kappa\) is the number of clusters, \(c_1<\cdots<c_{\kappa}\), and
\(\sum_\ell r_\ell=n\).  For \(r_\ell\geq2\). For \(r_\ell\geq2\), Fujie's theorem identifies an input atom pair giving
this output root, and we let \(Q_\ell\) be its tangent polynomial from
\Cref{thm:tangent-convolution}. The pair is unique. If two distinct pairs gave the same output root, then, the two sums being equal, both coordinates would differ and the corresponding groups of equal input roots would then be disjoint, so the four multiplicities would sum to at most \(2n\). But each pair producing an output atom has multiplicities summing to more than \(n\), which puts the total above \(2n\).

For a singleton cluster, define \(Q_\ell(z)=z-\lambda_\ell\),
where implicit differentiation gives
\(\lambda_\ell=-\left.\partial_\eps h_\eps(c_\ell)\right|_{\eps=0}/
h_0'(c_\ell)\). In degree one we set \(\operatorname{Disc}(Q_\ell)=1\),
this being the empty product of squared root gaps, and we assign no
degree-one value to \(\chi\) or \(\Phi\), since both need at least two
roots to be defined at all. 
It is worth saying in advance what the result will look like. Every pair of roots that collide contributes two to the vanishing order of the
discriminant, because each root gap enters it squared.  What survives in
the constant is then of two kinds, the tangent configurations themselves,
and the distances between distinct clusters.

\begin{proposition}[Discriminant at a collision]
\label{prop:collision-discriminant}
For the affine families \(p_\eps,q_\eps,h_\eps\) above, suppose
that each tangent polynomial \(Q_\ell\) is simple-rooted.  Put
\begin{equation}\label{eq:collision-discriminant-data}
 M=\sum_{\ell=1}^{\kappa} r_\ell(r_\ell-1),
 \qquad
 D_*=
 \prod_{\ell=1}^{\kappa}\operatorname{Disc}(Q_\ell)
 \prod_{\ell<j}|c_\ell-c_j|^{2r_\ell r_j}.
\end{equation}
Then \(D_*>0\), and there is a polynomial \(d_{\mathrm{reg}}\in\R[\eps]\)
with \(d_{\mathrm{reg}}(0)=D_*\) such that
\begin{equation}\label{eq:collision-discriminant-factorization}
 \operatorname{Disc}(h_\eps)=\eps^M d_{\mathrm{reg}}(\eps).
\end{equation}
In particular, the discriminant has vanishing order exactly \(M\) at
zero.  The hypothesis holds whenever
\(\operatorname{Disc}(h_\eps)\) is not identically zero.
\end{proposition}

\begin{proof}
The preceding theorem and the implicit-function theorem for the
singleton roots provide analytic branches
\(x_{\ell a}(\eps)=c_\ell+\eps\lambda_{\ell a}+O(\eps^2)\),
where \(\lambda_{\ell1},\ldots,\lambda_{\ell r_\ell}\) are
the distinct roots of \(Q_\ell\).  After shrinking the parameter
interval, these branches account for all \(n\) roots and are pairwise
distinct away from zero.  Within a cluster, each squared root gap is
\(\eps^2(\lambda_{\ell a}-\lambda_{\ell b}+O(\eps))^2\).
Between two clusters, it tends to \((c_\ell-c_j)^2\).  Multiplying
the within-cluster gaps gives the exponent \(M\) and the product of
the discriminants of the \(Q_\ell\).  There are \(r_\ell r_j\)
gaps between clusters \(\ell,j\), giving the remaining product in
\(D_*\).

Thus \(\eps^{-M}\operatorname{Disc}(h_\eps)\) extends analytically
to zero with value \(D_*>0\).  Since the discriminant itself is a
polynomial in \(\eps\), this says that it is divisible by
\(\eps^M\) in \(\R[\eps]\), and that the quotient has the
stated nonzero constant term.  The final assertion follows from
\Cref{lem:affine-tangent-simplicity}.
\end{proof}

The discriminant factorization separates the two contributions to entropy. Its logarithmic divergence counts the proportion of root pairs that collide, while its finite part also retains the distances between different clusters. For Fisher information, only the shrinking gaps contribute to the leading term, and their separation velocities determine its coefficient.

\begin{corollary}[Localization of information]
\label{cor:collision-information}
Under the hypotheses of \Cref{prop:collision-discriminant}, set
\(\theta=M/[n(n-1)]\).  Then
\begin{align}
 \chi_n[h_\eps]
   &=\theta\log|\eps|+
       \frac{\log D_*}{n(n-1)}+O(|\eps|),
       \label{eq:collision-entropy-localization}\\
 \cN_n(h_\eps)
   &=|\eps|^{2\theta}D_*^{\,2/[n(n-1)]}
           \bigl(1+O(|\eps|)\bigr).
       \label{eq:collision-entropy-power-localization}
\end{align}
Both \(\chi_n[h_\eps]-\theta\log|\eps|\) and
\(|\eps|^{-2\theta}\cN_n(h_\eps)\) extend real analytically
through zero.  The entropy's constant term may equivalently be written
\begin{equation}\label{eq:collision-entropy-constant}
 \frac{\log D_*}{n(n-1)}
 =\sum_{r_\ell\geq2}\frac{r_\ell(r_\ell-1)}{n(n-1)}
                      \chi_{r_\ell}[Q_\ell]
   +\frac{2}{n(n-1)}
       \sum_{\ell<j}r_\ell r_j\log|c_\ell-c_j|.
\end{equation}
Moreover, \(\eps^2\Phi_n(h_\eps)\) extends real analytically
through zero, with value
\begin{equation}\label{eq:collision-fisher-localization}
 K=\frac{4}{n(n-1)^2}
       \sum_{\ell=1}^{\kappa}\|s(\lambda_\ell)\|_2^2
   =\sum_{r_\ell\geq2}
       \frac{r_\ell(r_\ell-1)^2}{n(n-1)^2}
                           \Phi_{r_\ell}(Q_\ell),
\end{equation}
where \(\lambda_\ell\) is the root vector of \(Q_\ell\), and the
score of a singleton is the zero vector.  If at least one root
of \(h_0\) is repeated, then \(K>0\),
\(\Phi_n(h_\eps)=K\eps^{-2}+O(|\eps|^{-1})\), and
\(\Phi_n(h_\eps)^{-1}=K^{-1}\eps^2+O(|\eps|^3)\).
\end{corollary}

\begin{proof}
For a monic degree-\(n\) polynomial with simple real roots,
\(\chi_n=(n(n-1))^{-1}\log\operatorname{Disc}\).  The integer
\(M\) is even, and \(d_{\mathrm{reg}}\) from
\Cref{prop:collision-discriminant} is positive on an interval about
zero.  Taking the logarithm of \eqref{eq:collision-discriminant-factorization} proves \eqref{eq:collision-entropy-localization} and the analytic extension of its regular part. Exponentiating twice the regular part of the entropy gives the entropy-power expansion and shows that the normalized quantity \(|\varepsilon|^{-2\theta}\mathcal N_n(h_\varepsilon)\) extends analytically through zero. Taking the logarithm of the product defining \(D_*\) gives \eqref{eq:collision-entropy-constant}. This also explains why singleton clusters remain in the second sum even though they do not appear in the first.

Use the analytic branches from the preceding proof.  Each within-cluster
gap has the form \(\eps d_{\ell ab}(\eps)\), where \(d_{\ell ab}\) is
analytic with \(d_{\ell ab}(0)=\lambda_{\ell a}-\lambda_{\ell b}\ne0\),
while each gap between different clusters is analytic and already nonzero
at zero.  Now fix an output root \(x_{\ell a}(\eps)\) and multiply its
score by \(\eps\).  The terms coming from its own cluster become
reciprocals of analytic functions that do not vanish at zero, with limits
\((\lambda_{\ell a}-\lambda_{\ell b})^{-1}\), whereas the terms coming from the other clusters were bounded before the multiplication and so vanish after it. This is the whole of the localization where the score of a colliding root blows up at rate \(\eps^{-1}\), and only its own cluster is fast enough to be visible at that rate.  Hence \(\eps s(x(\eps))_{\ell a}\) extends analytically through zero with value \(s(\lambda_\ell)_a\).
Squaring these scaled score components and summing over all output roots, with the normalization in \(\eqref{eq:intro-fisher}\), gives the first expression for \(K\).  Replacing each cluster score norm by its degree-\(r_\ell\) Fisher information gives the second expression.

For a cluster of size \(r\geq2\), the identity
\(\sum_a\lambda_a s(\lambda)_a=\binom r2\) shows that its
score vector is nonzero.  Thus \(K>0\) precisely when some cluster
has size at least two.  In that case the analytic function \(\eps^2\Phi_n(h_\eps)\) has value \(K>0\) at zero and its Taylor expansion and its reciprocal give the final estimates.
\end{proof}

The Fisher weights in \eqref{eq:collision-fisher-localization} differ
from the entropy weights in \eqref{eq:collision-entropy-constant}
because Fisher information uses the degree-dependent score normalization
\(2/(n-1)\).  
The two quantities behave differently at a collision, and the difference is not a matter of normalization. Fisher information localizes by
\eqref{eq:collision-fisher-localization} where its leading coefficient is
assembled from the tangent polynomials alone, so a cluster that does not
collide contributes nothing to it and the positions of the clusters do not
enter.  On the other hand, entropy does not localize.  Its divergence rate \(\theta\) counts only colliding pairs, but the finite part
\eqref{eq:collision-entropy-constant} also carries the distances
\(|c_\ell-c_j|\) between distinct clusters, and therefore cannot be
recovered from the colliding groups alone. The example below exhibits both
halves of this: a surviving root that shifts the entropy constant while
contributing nothing to the Fisher coefficient.  What does come down to
normalization is the remaining discrepancy between the two sets of weights, since Fisher information carries the degree-dependent factor \(2/(n-1)\) that entropy does not.

\begin{example}[A double collision with a surviving root]
\label{ex:cubic-collision-geometry}
Let \(n=3\), and take
\(p_\eps=(x^2-\eps^2)(x-2)\) and
\(q_\eps=x^3-4\eps^2x\).  At zero the first polynomial has a
double root with velocities \(-1,1\), and the second has a triple
root with velocities \(-2,0,2\).  Their output collision therefore
has multiplicity two.  Its tangent polynomial is
\[
 Q(z)=(z^2-1)\boxplus_2
             \left(\frac13(z^3-4z)'\right)
      =(z^2-1)\boxplus_2(z^2-4/3)=z^2-7/3.
\]
The remaining input root gives \(c_{\mathrm{tan}}=-2\).  The coefficient formula
for convolution yields
\begin{equation}\label{eq:cubic-collision-geometry-polynomial}
 h_\eps(x)=x^3-2x^2-5\eps^2x+\frac{14}{3}\eps^2,
 \qquad
 \eps^{-2}h_\eps(\eps z)
       =-2(z^2-7/3)+\eps(z^3-5z).
\end{equation}
Thus the branches at zero are
\(x_\pm(\eps)=\pm\sqrt{7/3}\,\eps-\frac23\eps^2+O(\eps^3)\),
while the third root tends to \(2\).

There are two clusters, of sizes two and one. The tangent discriminant
is \(28/3\), and the four powers of the distance to the singleton
give \(D_*=(28/3)2^4=448/3\) (There are two gaps from the colliding pair to the singleton, and each tends to \(2\). Since each gap is squared in the discriminant, together they contribute the factor \(2^4\)).  Direct computation checks this
constant and also gives the normalized Fisher information:
\begin{align}
 \operatorname{Disc}(h_\eps)
   &=\frac43\eps^2(375\eps^4+264\eps^2+112),
       \label{eq:cubic-collision-geometry-discriminant}\\
 \Phi_3(h_\eps)
   &=\frac{(15\eps^2+4)^2}
          {2\eps^2(375\eps^4+264\eps^2+112)}.
       \label{eq:cubic-collision-geometry-fisher}
\end{align}
To verify the Fisher formula, we center the cubic. Translation preserves both the root gaps and the score. Let \(t_1,t_2,t_3\) be the roots of a
centered cubic \(f=x^3+Ax+B\).  Multiplying
\(\sum_{i<j}(t_i-t_j)^{-2}\) by \(\operatorname{Disc}(f)\)
gives \(\sum_k f'(t_k)^2=9A^2\), by Vieta's identities
\(\sum_k t_k^2=-2A\) and \(\sum_k t_k^4=2A^2\).
Expanding the score norm and cancelling the cross terms in triples gives
\(\|s(t)\|_2^2=2\sum_{i<j}(t_i-t_j)^{-2}\).  Consequently,
\(\Phi_3(f)=6A^2/\operatorname{Disc}(f)\).  Translating
\eqref{eq:cubic-collision-geometry-polynomial} to its mean gives
\(A=-4/3-5\eps^2\). Combining this Fisher identity with the discriminant expansion above gives the information asymptotics
\begin{align*}
 \chi_3[h_\eps]
   &=\frac13\log|\eps|+\frac16\log(448/3)+O(\eps^2),\\
 \cN_3(h_\eps)
   &=(448/3)^{1/3}|\eps|^{2/3}\bigl(1+O(\eps^2)\bigr),
 \qquad
 \Phi_3(h_\eps)\sim\frac1{14\eps^2}.
\end{align*}
The singleton root affects the entropy constant even though it does not
contribute to the leading Fisher term.  The exponent \(2/3\) in
entropy power records that one of the three pairs of output roots
collides.
\end{example}

For a finite free curve through simple input configurations, Fujie’s theorem gives a simple output at the base parameter. Its discriminant is therefore not identically zero, so \Cref{lem:affine-tangent-simplicity} gives distinct tangent roots at every real collision. We can now use these local descriptions to study the real singularities of the whole curve and the root permutations induced by projective continuation.

\section{The real geometry of finite free curves}
\label{sec:collision-geometry}
The tangent formula explains how a multiple output root separates.  We now
ask where on a finite free curve such a collision can occur at all, and how the branches fit together once it has. A plane curve of degree \(n\) has no reason to be well behaved in either respect. Its singular fibers may be many, its singularities may be complicated, and its real branches may be distributed almost arbitrarily. Finite free curves are not like that. For affine input trajectories through simple configurations there is at most one real collision parameter, whatever the degree, every real singularity is an ordinary totally real multiple point, and the covering of the real projective line obtained by normalizing the curve is computed from the ordered collision multiplicities and from nothing else. The rigidity comes from the atom theorem, which is a statement about multiplicities of a convolution and carries no geometric content at all until it is put on a pencil.

Throughout this section, \(n\geq2\) and \(\alpha,\beta\in\R^n\)
are fixed simple root vectors.  For \(w=u\oplus v\), write
\[
 h_w(x,y,t)=\mathcal F_n(x,y\alpha+tu,y\beta+tv),
 \qquad C_w=\{h_w=0\}\subset\mathbf P^2_{\R}.
\]
We retain the split subspace
\begin{equation}\label{eq:global-collision-split-space}
 L_{\alpha,\beta}
 =\{(a\one+c\alpha)\oplus(b\one+c\beta):a,b,c\in\R\}.
\end{equation}
Using \(\alpha^\circ,\beta^\circ\) in this formula gives the same
subspace, since their means can be absorbed into \(a,b\).  By
\Cref{thm:intro-split-locus}, this is precisely the locus of completely
split curves.

\subsection{Where real singularities can occur}

\Cref{lem:unique-real-collision-fiber}, proved on the way to the split
classification, already places every real collision over a single
parameter, so there is only one fiber to examine.  What remains is to
identify the branches through it, and they turn out to be as simple as
branches can be. 
We use the term \emph{ordinary \(r\)-fold point} for a plane-curve
singularity whose lowest-degree homogeneous term is a product of \(r\)
distinct linear forms.  It is \emph{totally real} when these forms can all
be chosen over \(\R\).  The next theorem identifies the tangent forms,
as well as the singularity type.

The base fiber at \([1:0]\) is simple, so every singular real fiber has the form \([-c:1]\). Near such a fiber, we work in the chart \(t=1\) and use the coordinate \(Y=y+c\). The input roots then become \(u_i-c\alpha_i+Y\alpha_i\) and \(v_j-c\beta_j+Y\beta_j\). Thus the original root coordinates \(\alpha_i,\beta_j\) are the velocities in this local description. This explains their appearance in the tangent convolution below.

\begin{theorem}[Ordinary real singularities]
\label{thm:ordinary-real-singularities}
Every real singular point of \(C_w\) is an ordinary, totally real
multiple point.  All these points lie over the same real projective
parameter, if any exist.

More precisely, let \([-c:1]\) be a parameter at which an output root
\(a+b\) has multiplicity \(r\geq2\).  Let \(I,J\) be the maximal
input blocks at \(a,b\), respectively, for the configurations
\(u-c\alpha\) and \(v-c\beta\).  Put \(m=|I|\) and \(k=|J|\),
so that \(r=m+k-n\).  Define the monic degree-\(r\) polynomial
\begin{equation}\label{eq:curve-collision-tangent-polynomial}
 Q_{I,J}(z)=
 \left(\frac{r!}{m!}\frac{d^{m-r}}{dz^{m-r}}P_{\alpha_I}(z)\right)
 \boxplus_r
 \left(\frac{r!}{k!}\frac{d^{k-r}}{dz^{k-r}}P_{\beta_J}(z)\right),
\end{equation}
where \(\alpha_I=(\alpha_i)_{i\in I}\) and similarly for
\(\beta_J\).  Its roots \(\zeta_1<\cdots<\zeta_r\) are real and
distinct.  In the chart \(t=1\), the tangent cone at
\([a+b:-c:1]\) is a nonzero scalar multiple of
\(\prod_{\ell=1}^r(X-\zeta_\ell Y)\), where
\(X=x-a-b\) and \(Y=y+c\).
\end{theorem}

\begin{proof}
The fiber at \([1:0]\) is simple, so any repeated real fiber has
\(t\ne0\) and can be written as \([-c:1]\).  Every repeated output
root is an atom sum by Fujie's theorem, with the multiplicities stated
above.  In the local parameter \(Y=y+c\), the roots in the two input
blocks are exactly \(a+Y\alpha_i\) and \(b+Y\beta_j\).
\Cref{thm:tangent-convolution} therefore gives
\begin{equation}\label{eq:curve-collision-initial-form}
 h_w(a+b+X,-c+Y,1)
 =c_{\mathrm{tan},0}Y^rQ_{I,J}(X/Y)
   +\text{terms of total degree at least }r+1,
\end{equation}
with
\begin{equation}\label{eq:curve-collision-leading-constant}
 c_{\mathrm{tan},0}=\frac{m!k!}{n!r!}
 \prod_{i\notin I}(a-u_i+c\alpha_i)
 \prod_{j\notin J}(b-v_j+c\beta_j).
\end{equation}
Maximality of the blocks makes every factor in these products nonzero.
At \(Y=0\), we use the homogeneous polynomial extension of the
tangent expression.

Both \(P_{\alpha_I}\) and \(P_{\beta_J}\) have simple real roots.
Their normalized derivatives in
\eqref{eq:curve-collision-tangent-polynomial} are simple by Rolle's
theorem, and their convolution is simple by Fujie's regularity theorem.
Thus the initial form in \eqref{eq:curve-collision-initial-form} has
\(r\) distinct real factors. Substitute \(X=Yz\) and divide the equation by \(Y^r\). The resulting equation is analytic at \(Y=0\), where it reduces to \(c_{\mathrm{tan},0}Q_{I,J}(z)=0\). Since every root \(\zeta_\ell\) is simple, the implicit-function theorem gives an analytic solution \(z_\ell(Y)\) with \(z_\ell(0)=\zeta_\ell\). Returning to the original coordinate \(X\) gives the local branches

\begin{equation}\label{eq:ordinary-collision-branches}
 X_\ell(Y)=Y\zeta_\ell+O(Y^2),\qquad 1\leq\ell\leq r.
\end{equation}
In particular, the point is singular and ordinary of multiplicity \(r\). Conversely, a singular point must have \(\partial_xh_w=0\), and hence
must be a repeated root of its fiber.  This accounts for all real
singular points. \Cref{lem:unique-real-collision-fiber} places them over
a single parameter. 
\end{proof}

The tangent polynomial in \eqref{eq:curve-collision-tangent-polynomial} depends only on the original roots in the two blocks.  The other velocity coordinates affect the higher-order terms and the nonzero factor \(c_{\mathrm{tan},0}\). Thus the local convolution of \Cref{sec:collision-local} also fixes the directions in which the branches meet.  In particular, the entropy and information limits in \Cref{cor:collision-information} are attached to explicit lower-degree convolutions at these points.

The cubic of \Cref{ex:cubic-collision-geometry} fits the present setting because both inputs are simple at \(\varepsilon=1\). We take that parameter as the base, while retaining \(\varepsilon=0\) as the location of the node. The figure shows how the two branches separate there. At the collision, the two root positions coincide, but their first-order velocities remain distinct. The rescaling \(z=x/\eps\) retains those velocities where after removing the factor \(\eps^2\), the equation extends to \(\eps=0\), where it becomes \(Q(z)=0\). Thus the two branches
pass through different points in the rescaled chart.  For a collision
of multiplicity \(r\), the same argument gives \(r\) distinct real
points, one for each root of the tangent convolution.

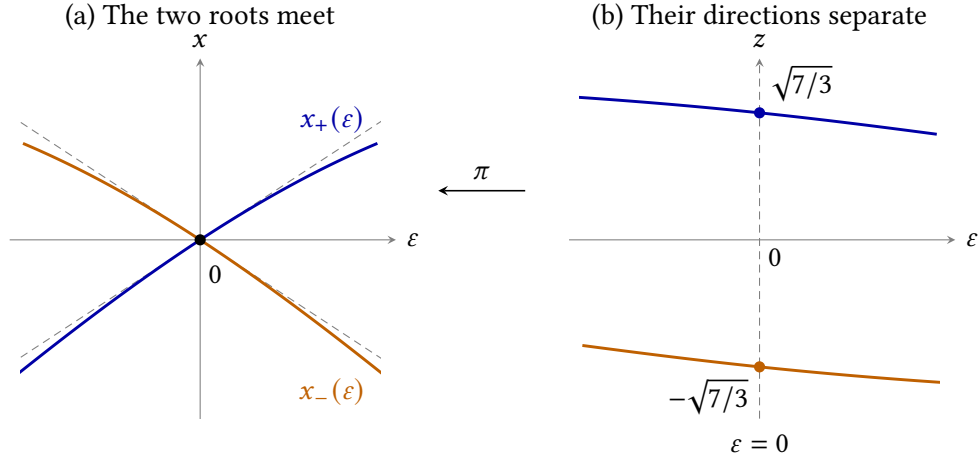
\begin{figure}[!htbp]
\centering
\begin{tikzpicture}[font=\small,>=stealth,
  rootplus/.style={blue!65!black,line width=1.1pt},
  rootminus/.style={orange!75!black,line width=1.1pt},
  axis/.style={gray,thin,->},
  tangent/.style={gray,densely dashed,thin}]
  % Exact parametrization of the strict transform, with s=z:
  % e(s)=2(s^2-7/3)/(s^3-5s), x(s)=s e(s).
  % The two parameter intervals avoid all poles of e(s).
  \pgfmathdeclarefunction{collisiontime}{1}{%
    \pgfmathparse{2*(#1*#1-7/3)/(#1*#1*#1-5*#1)}%
  }
  \node at (0,2.95) {(a) The two roots meet};
  \begin{scope}[x=7cm,y=3cm]
    \draw[axis] (-0.36,0) -- (0.37,0)
      node[right,text=black] {\(\eps\)};
    \draw[axis] (0,-0.79) -- (0,0.80)
      node[above,text=black] {\(x\)};
    \begin{scope}
      \clip (-0.34,-0.73) rectangle (0.34,0.73);
      \draw[tangent] (-0.36,{-0.36*sqrt(7/3)})
        -- (0.36,{0.36*sqrt(7/3)});
      \draw[tangent] (-0.36,{0.36*sqrt(7/3)})
        -- (0.36,{-0.36*sqrt(7/3)});
      \draw[rootplus] plot[variable=\s,domain=1.27:1.77,samples=121]
        ({collisiontime(\s)},{\s*collisiontime(\s)});
      \draw[rootminus] plot[variable=\s,domain=-1.77:-1.27,samples=121]
        ({collisiontime(\s)},{\s*collisiontime(\s)});
    \end{scope}
    \fill (0,0) circle[radius=2.1pt];
    \node[right,inner sep=3pt] at (0,-0.15) {\(0\)};
    \node[text=blue!65!black] at (0.25,0.53) {\(x_+(\eps)\)};
    \node[text=orange!75!black] at (0.25,-0.67) {\(x_-(\eps)\)};
  \end{scope}
  \draw[->,semithick] (4.30,0.65) -- (3.15,0.65)
    node[midway,above] {\(\pi\)};
  \node at (7.4,2.95) {(b) Their directions separate};
  \begin{scope}[shift={(7.4,0)},x=7cm,y=1.10cm]
    \draw[axis] (-0.36,0) -- (0.37,0)
      node[right,text=black] {\(\eps\)};
    \draw[axis,densely dashed] (0,-2.15) -- (0,2.18)
      node[above,text=black] {\(z\)};
    \begin{scope}
      \clip (-0.34,-2.08) rectangle (0.34,2.08);
      \draw[rootplus] plot[variable=\s,domain=1.27:1.77,samples=121]
        ({collisiontime(\s)},\s);
      \draw[rootminus] plot[variable=\s,domain=-1.77:-1.27,samples=121]
        ({collisiontime(\s)},\s);
    \end{scope}
    \fill[blue!65!black] (0,{sqrt(7/3)}) circle[radius=2.1pt];
    \fill[orange!75!black] (0,{-sqrt(7/3)}) circle[radius=2.1pt];
    \node[above right,inner sep=4pt] at (0,{sqrt(7/3)})
      {\(\sqrt{7/3}\)};
    \node[below left,inner sep=4pt] at (0,{-sqrt(7/3)})
      {\(-\sqrt{7/3}\)};
    \node[below right,inner sep=3pt] at (0,0) {\(0\)};
    \node[below] at (0,-2.20) {\(\eps=0\)};
  \end{scope}
\end{tikzpicture}
\caption{The ordinary real node of
\Cref{ex:cubic-collision-geometry}, illustrating
\Cref{thm:ordinary-real-singularities}.  The left panel shows the two
analytic root branches and their dashed tangent lines.  Each branch
keeps its colour through the crossing, while their order reverses.
In the rescaled chart on the right, the marked points are the two roots
of the tangent convolution \(Q(z)=z^2-7/3\).  The map \(\pi(\eps,z)=(\eps,\eps z)\) sends both points to the node. The curves are drawn from the exact cubic equation and the third root, which tends to \(2\), is outside this local picture.}
\label{fig:ordinary-node-blowup}
\end{figure}

An output root of multiplicity at least \(r\) requires two input atoms
containing at least \(n+r\) labels between them, so at most \(n-r\) of the
\(2n\) input labels can remain outside those atoms. On the blocks that do
collide, the velocities are forced to agree with two translations and one
common dilation and only the few exceptional coordinates are free to depart from that motion. The condition is therefore linear, which is the point that the singularity locus can be written down directly in the space of input directions, instead of being found by eliminating the roots.

\begin{theorem}[The real collision arrangement]
\label{thm:collision-arrangement}
For every \(2\leq r\leq n\),
\begin{equation}\label{eq:collision-arrangement}
 \Sigma_r=
 \bigcup_{|S|=n-r}\bigl(L_{\alpha,\beta}+\R^S\bigr).
\end{equation}
The \(\binom{2n}{n-r}\) subspaces in this union are distinct and have
dimension \(n-r+3\).  They are precisely the irreducible components of
\(\Sigma_r\) as a reduced real algebraic set.  In particular,
\(\Sigma_2\) has codimension \(n-1\), and
\(\Sigma_n=L_{\alpha,\beta}\).
\end{theorem}

\begin{proof}
Suppose first that \(w\in L_{\alpha,\beta}+\R^S\), where
\(|S|=n-r\).  Outside the exceptional labels in \(S\), write
\(u_i=a+c\alpha_i\) and \(v_j=b+c\beta_j\).  At
\([-c:1]\), these labels give atoms of multiplicities at least
\(m=n-|S\cap\{\text{first-input labels}\}|\) and
\(k=n-|S\cap\{\text{second-input labels}\}|\).  Since
\(m+k=n+r\), their sum is an output root of multiplicity at least
\(r\).  Thus \(w\in\Sigma_r\).

Conversely, take a repeated root of multiplicity \(\rho\geq r\).
As in \Cref{thm:ordinary-real-singularities}, its parameter is
\([-c:1]\), and its input blocks satisfy
\(u_i=a+c\alpha_i\) on \(I\) and
\(v_j=b+c\beta_j\) on \(J\), with
\(|I|+|J|=n+\rho\).  There are \(n-\rho\leq n-r\)
exceptional coordinates outside these blocks.  Enlarge their set to a
set \(S\) of size \(n-r\).  Then
\(w\in L_{\alpha,\beta}+\R^S\), proving the equality of sets.

It remains to check that the displayed subspaces have the asserted
dimensions and are distinct.  A nonzero vector in \(L_{\alpha,\beta}\)
has at least \(n\) nonzero coordinates.  Indeed, when \(c=0\), a
nonzero constant block already has \(n\) entries.  When \(c\ne0\),
simplicity allows at most one zero in each block, so the support has size
at least \(2n-2\geq n\).  Since \(|S|\leq n-2\), it follows that
\(L_{\alpha,\beta}\cap\R^S=0\).  The dimension is therefore
\(3+|S|=n-r+3\).

If two subspaces for distinct sets \(S,T\) were equal, choose
\(i\in S\setminus T\).  An expression
\(e_i=\ell+z\), with \(\ell\in L_{\alpha,\beta}\) and
\(z\in\R^T\), has \(\ell\ne0\), since \(i\notin T\).
It would therefore give a nonzero vector in \(L_{\alpha,\beta}\)
supported on at most \(|T|+1\leq n-1\) coordinates.  This contradicts the support bound.
The subspaces are distinct, and none contains another because their
dimensions agree.  This proves the assertion about irreducible
components.
\end{proof}

\Cref{thm:ordinary-real-singularities,thm:collision-arrangement} together
prove \Cref{thm:intro-real-geometry} where the first says what the real
singularities look like, the second says exactly which directions produce
them. The classification is set-theoretic and over \(\R\) where the locus \(\Sigma_2\) records the real singular points of the curves, and it says nothing about singular points lying over non-real parameters, and \eqref{eq:collision-arrangement} is not a scheme-theoretic identity.  What the stratification does give is a picture with a fixed exchange rate. Each step down in the required output multiplicity buys exactly one further input-root coordinate that may leave the common affine motion, and the chain terminates at \(\Sigma_n=L_{\alpha,\beta}\), the split locus from the equality proof. The equality problem of \Cref{sec:information-equality} therefore sits at the very top of a stratification whose lower strata occupy the rest of this section.

\begin{example}[A reducible curve without real singular points]
\label{ex:reducible-real-smooth-finite-free-curve}
Take \(\alpha=\beta=(-1,0,1)\), \(u=\alpha\), and \(v=0\).
The two input polynomials are odd cubics, and their convolution is
\[
 h_w(x,y,t)=x\bigl(x^2-(y+t)^2-y^2\bigr).
\]
The quadratic form \((y+t)^2+y^2\) is positive at every nonzero real
parameter, so every real fiber has three distinct roots.  The curve is
nevertheless reducible.  Its line and conic intersect at two non-real points,
where \(x=0\) and \((y+t)^2+y^2=0\).  Thus the complement of
\(\Sigma_2\) contains curves with complex singularities and curves
with more than one irreducible component.
\end{example}

\subsection{Following the real branches}

At an ordinary multiple point, normalization separates the branches \(X_\ell(Y)=Y\zeta_\ell+O(Y^2)\). On each branch, \(Y\) remains a local coordinate, so projection to the parameter line is unramified even above the singular fiber. The real normalization therefore forms a covering of the real projective line. We determine its connected components by following the roots through one circuit of that line.
The projection to \([y:t]\) then becomes a covering of real circles,
even at the collision parameter.  In the terminology of
\cite{KummerShamovich20}, it is a \emph{real-fibered morphism}: a
complex point has real image if and only if it is real.  On each
irreducible normalized component, it is a separating morphism as studied
in~\cite{KummerShaw20}.  The degrees of its real connected components
can be read from the ordered collision multiplicities.

If a singular real fiber exists, list its distinct output roots in
increasing order in a real representative of the parameter and write
\((r_1,\ldots,r_{\kappa})\) for their multiplicities, including the
singletons.  Put \(N_j=r_1+\cdots+r_j\), with \(N_0=0\), and define
two permutations of \(\{1,\ldots,n\}\) by
\begin{equation}\label{eq:collision-reversal-permutations}
 \sigma_{\mathrm{rev}}(i)=n+1-i,
 \qquad
 \sigma_{\mathrm{blk}}(i)=N_{j-1}+N_j+1-i
 \quad\text{if }N_{j-1}<i\leq N_j.
\end{equation}
Thus \(\sigma_{\mathrm{rev}}\) reverses the whole list and \(\sigma_{\mathrm{blk}}\) reverses each collision
block.  If no singular real fiber exists, take \(\sigma_{\mathrm{blk}}\) to be the identity.

\begin{theorem}[The real normalization covering]
\label{thm:real-monodromy}
The curve \(C_w\) is geometrically reduced.  Let
\(\nu:\widetilde C_w\to C_w\) be its normalization, with disjoint
components when \(C_w\) is reducible.  The morphism
\begin{equation}\label{eq:normalized-real-projection}
 \pi:\widetilde C_w\longrightarrow\mathbf P^1_{\R},
 \qquad \pi=([x:y:t]\mapsto[y:t])\circ\nu,
\end{equation}
is finite of degree \(n\), real-fibered, and unramified at its real
points.  The monodromy of
\(\widetilde C_w(\R)\to\mathbf P^1(\R)\) is conjugate to \(\sigma_{\mathrm{rev}}\sigma_{\mathrm{blk}}\).
Consequently, its connected components correspond to the cycles of
\(\sigma_{\mathrm{rev}}\sigma_{\mathrm{blk}}\), and their covering degrees are the cycle lengths.
If \(\sigma_{\mathrm{rev}}\sigma_{\mathrm{blk}}\) is a single cycle, then \(C_w\) is absolutely
irreducible.
\end{theorem}

\begin{proof}
The polynomial \(h_w\) is homogeneous of degree \(n\) and monic of
degree \(n\) in \(x\).  Every irreducible factor over \(\C\) can
therefore be normalized to be homogeneous and monic in \(x\), with
positive \(x\)-degree.  A repeated factor would give a repeated factor
of \(h_w(x,1,0)\), contrary to the simplicity of that fiber.  This
proves geometric reducedness.

The point \([1:0:0]\) does not lie on \(C_w\), so projection defines a projective morphism to \(\mathbf P^1\).  Each fiber is finite, because its equation is monic in \(x\) and the morphism is therefore finite. Normalization is finite as well, and the simple base fiber
shows that the degree of their composition is \(n\).

Hyperbolicity makes every point of \(C_w\) over a real parameter
real.  At a smooth point of such a fiber, \(\partial_xh_w\ne0\) by
\Cref{thm:ordinary-real-singularities},
so projection is a local analytic isomorphism.  At a singular real
point, \eqref{eq:ordinary-collision-branches} gives all its normalized
branches.  They are real, and \(Y\) is a local coordinate on each of
them.  Hence every normalized point above a real parameter is real,
and projection is unramified there.  Since the real normalization is
a compact smooth one-dimensional manifold, it is a disjoint union of
circles, each covering \(\mathbf P^1(\R)\).  This also gives directly
the conclusion supplied in general by \cite[Theorem~2.19]{KummerShamovich20}.

A circuit of the real projective line can be represented by \((y,t)=(\cos\theta,\sin\theta)\), with \(\theta\) varying over an interval of length \(\pi\) whose endpoints lie above a nonsingular fiber. On each nonsingular interval, the roots retain their order. Near a multiple root, the expansion
\(X_\ell(Y)=Y\zeta_\ell+O(Y^2)\) shows that their order for \(Y<0\)
is the reverse of their order for \(Y>0\). Continuation through the singular fiber therefore induces the block reversal \(\sigma_{\mathrm{blk}}\). If there is no singular real fiber, this permutation is the identity.

At the end of the half-circle, homogeneity identifies the output
roots with the negatives of those at the beginning. The projective
identification \([x:-y:-t]=[-x:y:t]\) reverses the complete ordered
list and contributes \(\sigma_{\mathrm{rev}}\). The monodromy is
therefore \(\sigma_{\mathrm{rev}}\sigma_{\mathrm{blk}}\), up to
relabeling and choice of orientation. Its cycles correspond to the
connected components of the real normalization. A cycle of length
\(d\) means that a point returns to its starting position after
\(d\) circuits of the parameter line, so the corresponding component
has covering degree \(d\).

Finally, every complex irreducible factor of \(h_w\) is defined over
\(\R\).  For real parameters its roots are a subset of the real
roots of \(h_w\) and the coefficients of its monic normalization are
therefore real-valued on \(\R^2\), and hence are real polynomials.
Moreover, each factor contributes at least one point of the simple base
fiber, so its normalization has a real connected component.  A single
cycle in the full covering is thus possible only if there is a single
complex irreducible factor.
\end{proof}

In degree four, the full reversal is \(\sigma_{\mathrm{rev}}=(14)(23)\). For the composition \((2,1,1)\), the block reversal is \((12)\), so their product is the four-cycle \((1324)\). For \((1,2,1)\), the block reversal is \((23)\), and the product is \((14)\), with \(2\) and \(3\) fixed. Moving the double block from the edge of the ordered fiber to its middle therefore changes one real component of degree four into three components of degrees \(2,1,1\).

These are statements about the covering of the real normalization.
They leave open the complex genus and, when there is more than one real
component, the irreducible decomposition.  In particular, more than one
cycle does not imply reducibility.

\begin{corollary}[Realization of the covering types]
\label{cor:collision-composition-realization}
For every fixed simple pair \(\alpha,\beta\), every ordered
composition \((r_1,\ldots,r_{\kappa})\) of \(n\) occurs as the ordered
multiplicities of a real fiber of a finite free curve through that pair.
If some \(r_j\geq2\), it is the unique singular real fiber. For the
all-singleton composition the curve can be chosen with no singular real
fiber.  Thus the possible covering types in
\Cref{thm:real-monodromy} are independent of the fixed base pair and
are exactly the cycle types obtained from these block reversals.
\end{corollary}

\begin{proof}
Partition the second-input labels into sets \(J_1,\ldots,J_{\kappa}\) of
the prescribed sizes, and choose real numbers \(b_1<\cdots<b_{\kappa}\).
Set \(u=\alpha\) and \(v_i=\beta_i+b_j\) for \(i\in J_j\).
At \([y:t]=[-1:1]\), the first input is \(x^n\), and the second
is \(\prod_j(x-b_j)^{r_j}\).  Their convolution is the second
polynomial itself.  At every other parameter, the first input has the
distinct roots \((y+t)\alpha_i\), so the output is simple.  The
fiber at \([1:0]\) is the prescribed base pair throughout.
If centered directions are desired, subtract the two velocity means and this only translates each output fiber and preserves the covering type.
\end{proof}

\begin{corollary}[Two collision blocks]
\label{cor:two-block-collision-irreducibility}
Let \(n\geq3\).  If a real fiber has exactly two distinct output roots
of multiplicities \(r,s\), where \(r+s=n\), then the real
normalization has \(\gcd(r,s)\) connected components.  Each covers
\(\mathbf P^1(\R)\) with degree \(n/\gcd(r,s)\).  In particular,
if \(r,s\) are coprime, then \(C_w\) is absolutely irreducible.
\end{corollary}

\begin{proof}
Since \(n\geq3\), this fiber is singular and hence is the unique
singular real fiber.  For the composition \((r,s)\), direct
substitution into \eqref{eq:collision-reversal-permutations} shows
that \(\sigma_{\mathrm{rev}}\sigma_{\mathrm{blk}}\) is rotation by \(s\) modulo \(n\).  It has
\(\gcd(n,s)=\gcd(r,s)\) cycles of the stated length.  Apply
\Cref{thm:real-monodromy}.
\end{proof}

For example, a degree-five finite free curve with one triple and one
double output root in a real fiber is absolutely irreducible.  Its real
normalization is a single circle covering the parameter circle five
times.  \Cref{cor:collision-composition-realization} constructs such a
curve through any prescribed pair of simple degree-five inputs.

\subsection{The non-real part of the discriminant}

The restrictions just proved on real collisions have a consequence that is
not about collisions at all. Because so little of the discriminant is able to vanish over the real parameter line, a definite quantity of it is forced off that line, and the quantity can be computed exactly. Define the binary form \(D_w(y,t)=\operatorname{Disc}_x h_w(x,y,t)\). The Homogeneity gives it degree \(n(n-1)\), and the simple base fiber makes it nonzero, so its projective zeros can be counted with their algebraic multiplicities.

\begin{corollary}[A sharp non-real discriminant gap]
\label{cor:non-real-discriminant-gap}
Let \(n\geq3\).  If \(w\notin L_{\alpha,\beta}\), then \(D_w\)
has at least \(2n-2\) non-real projective zeros, counted with
multiplicity.  Equality holds if and only if
\(w\in\Sigma_{n-1}\setminus L_{\alpha,\beta}\).
In this case \(C_w\) is absolutely irreducible, its normalization is
\(\mathbf P^1_{\R}\), and its only complex singular point is the
ordinary \((n-1)\)-fold point in its real collision fiber.
\end{corollary}

\begin{proof}
If there is a singular real fiber, let \((r_1,\ldots,r_{\kappa})\) be its
ordered composition, including singletons.  By
\eqref{eq:ordinary-collision-branches}, the gap between two branches
in a block of size \(r_j\) vanishes to first order in a local
parameter.  Gaps between different blocks remain nonzero.  The
discriminant therefore vanishes to order
\begin{equation}\label{eq:real-discriminant-zero-order}
 \operatorname{ord}D_w=\sum_{j=1}^{\kappa} r_j(r_j-1).
\end{equation}
This is also the local discriminant factorization in
\Cref{prop:collision-discriminant}.  There are no other real zeros by
\Cref{lem:unique-real-collision-fiber}.  If no singular real fiber
exists, the total real zero order is zero.

Since \(w\notin L_{\alpha,\beta}=\Sigma_n\), the composition
cannot be \((n)\).  Among compositions of \(n\) into at least two
parts, the largest value of \(\sum_jr_j(r_j-1)\) is
\((n-1)(n-2)\), attained only at \((n-1,1)\) and its reversal.
Indeed, merging two parts increases the sum by twice their product and after reducing to two parts, moving a unit from the smaller to the
larger increases the sum until the smaller is one.  Subtracting from
the total degree gives
\[
 n(n-1)-(n-1)(n-2)=2n-2.
\]
For \(n\geq3\), equality requires a singular real fiber of type
\((n-1,1)\), which is precisely
\(\Sigma_{n-1}\setminus L_{\alpha,\beta}\).

In the equality case, \Cref{cor:two-block-collision-irreducibility}
makes the curve absolutely irreducible.  An ordinary \((n-1)\)-fold
point has \(\delta\)-invariant \(\binom{n-1}{2}\): its smooth
branches intersect pairwise transversely, each contributing one.  This is
the arithmetic genus of a plane curve of degree \(n\).  The
normalization genus formula
\(g(\widetilde C_w)=\binom{n-1}{2}-\sum_p\delta_p\), with the sum
over complex singular points, therefore forces genus zero and excludes
any further complex singularity (see~\cite[Chapter~IV]{Hartshorne77}).
The normalization has real points above the simple base fiber. A smooth
projective genus-zero real curve with a real point is isomorphic to
\(\mathbf P^1_{\R}\), completing the proof.
\end{proof}

The bound is attained through every fixed base pair.  For example, take
a direction with exactly one nonzero coordinate, in the second input.  At \([y:t]=[0:1]\), the first input collapses completely
and the second has one exceptional root, so the output has
multiplicities \((n-1,1)\).  Such a direction is outside
\(L_{\alpha,\beta}\), by the support bound in the proof of
\Cref{thm:collision-arrangement}.

For \(n=2\), a non-split curve has no singular real fiber, since
\(\Sigma_2=L_{\alpha,\beta}\).  Its real monodromy is the
transposition \(\sigma_{\mathrm{rev}}\), so it is absolutely irreducible and therefore
a smooth conic.  It has real points and is rational, and both zeros
of its degree-two discriminant are non-real.  Thus the numerical bound
\(2n-2\) remains sharp, but its equality case has no singular point.

The cubic in \Cref{ex:cubic-collision-geometry} illustrates the equality
case in degree three.  Its homogeneous equation and discriminant are
\begin{align*}
 H(x,y,t)&=x^3-2x^2y-5t^2x+\frac{14}{3}t^2y,\\
 \operatorname{Disc}_xH&=\frac43t^2
       \bigl(375t^4+264t^2y^2+112y^4\bigr).
\end{align*}
Both input configurations are simple at \([y:t]=[1:1]\) and choosing
that point as the base puts the curve in the setting of this section.
The quartic factor is positive on the real projective parameter line,
so the node at \([0:1:0]\) is its only real collision and there are
four non-real discriminant zeros.  Its composition is \((2,1)\), for
which \(\sigma_{\mathrm{rev}}\sigma_{\mathrm{blk}}=(13)(12)\) is a three-cycle.  Thus the real normalization
is connected, and the cubic is irreducible and rational.  The two
branches whose slopes computed the Fisher pole in that example also
determine this global covering.

The count in \Cref{cor:non-real-discriminant-gap} is a count of discriminant zeros, and a zero of \(D_w\) is a parameter value that can arise in either of two ways. Either from a singular point of the curve, an intersection of distinct components included, or from ramification of the normalized projection. In the rational equality case the two sources separate cleanly, because the ordinary real \((n-1)\)-fold point is then the only singularity. Every remaining zero must therefore come from non-real ramification, and their total multiplicity is exactly \(2n-2\).

\subsection{Entropy detects the equality direction}
\label{sec:entropy-detects-equality}

We can now return to the equality question with which the paper began.
The discriminant counts colliding pairs of roots, while entropy takes
its logarithm. The local formula makes this relation precise enough
to distinguish complete splitting from every partial collision.

At the unique real singular fiber, each pair of roots that meets contributes two to the discriminant’s vanishing order. Dividing by \(n(n-1)\), twice the total number of root pairs, therefore gives the fraction of pairs that collide. The homogeneous discriminant makes this definition independent of the affine chart and includes intersections at infinity. Recall that \(D_w(y,t)=\operatorname{Disc}_x h_w(x,y,t)\) is a
nonzero binary form of degree \(n(n-1)\).  Define its normalized
total real zero order by
\begin{equation}\label{eq:real-entropy-order}
 \vartheta(w)=\frac{1}{n(n-1)}
    \sum_{\xi\in\mathbf P^1(\R)}\operatorname{ord}_{\xi}D_w.
\end{equation}
The sum is zero when there is no real collision.  Otherwise there is
only one real collision parameter, by
\Cref{lem:unique-real-collision-fiber}.  Choose a local affine chart
there, with coordinate \(\eps\) vanishing at that parameter, and
let \(h_\eps\) be the corresponding monic polynomial in \(x\).
\Cref{cor:collision-information} gives
\begin{equation}\label{eq:entropy-order-local-meaning}
 \chi_n[h_\eps]=\vartheta(w)\log|\eps|+O(1),
 \qquad
 \cN_n(h_\eps)=|\eps|^{2\vartheta(w)}
                    \bigl(c_\xi+O(|\eps|)\bigr),
 \quad c_\xi>0.
\end{equation}
Changing the local parameter or the nonvanishing projective lift
changes the constant term, but not \(\vartheta(w)\). The fact that the definition is projective is actually important here. For a pure translation \(w=a\mathbf1\oplus b\mathbf1\), the output roots in the chart \(y=1\) are \(\gamma_i+t(a+b)\), where \(\gamma=\Omega(\alpha,\beta)\). Their gaps remain constant at every finite time. In the chart \(t=1\), however, the same roots are \(a+b+y\gamma_i\), and they all meet at \(y=0\). An affine definition would not merely overlook that intersection, it would record the translation direction as having no collision at all. The projective
definition sees it.

Suppose \(n\geq3\).  The split classification and the non-real
discriminant gap give
\begin{align}
 \vartheta(w)=1
 &\quad\Longleftrightarrow\quad w\in L_{\alpha,\beta},
       \label{eq:entropy-order-split}\\
 w\notin L_{\alpha,\beta}
 &\quad\Longrightarrow\quad
       0\leq\vartheta(w)\leq\frac{n-2}{n},
       \label{eq:entropy-order-gap}\\
 \vartheta(w)=\frac{n-2}{n}
 &\quad\Longleftrightarrow\quad
       w\in\Sigma_{n-1}\setminus L_{\alpha,\beta}.
       \label{eq:entropy-order-rational-stratum}
\end{align}
Indeed, if the collision multiplicities are
\((r_1,\ldots,r_{\kappa})\), then
\(n(n-1)\vartheta(w)=\sum_jr_j(r_j-1)\), including singleton
blocks.  This sum reaches \(n(n-1)\) only for the composition
\((n)\), whose locus is
\(\Sigma_n=L_{\alpha,\beta}\).  The remaining assertions are
\Cref{cor:non-real-discriminant-gap} written in terms of real zero
order.  In particular, the largest value below one identifies curves
whose normalization is \(\mathbf P^1_{\R}\) and whose only
complex singularity is an ordinary \((n-1)\)-fold point.

Every input pair admits a common dilation, so the value one for an
arbitrary direction does not distinguish Hermite inputs.  The direction
must be chosen by the information problem itself.

\begin{samepage}
\begin{corollary}[An entropy criterion for Stam equality]
\label{cor:score-entropy-equality}
Let \(n\geq3\), and let \(\alpha,\beta\in\R^n\) have simple
real coordinates.  Put
\begin{equation}\label{eq:optimal-score-direction}
 I_\alpha=\|s(\alpha)\|_2^2,\qquad
 I_\beta=\|s(\beta)\|_2^2,\qquad
 w_* = \frac{s(\alpha)}{I_\alpha}
          \oplus\frac{s(\beta)}{I_\beta}.
\end{equation}
Then \(\vartheta(w_*)=1\) if and only if equality holds in the
finite free Stam inequality for \(P_\alpha,P_\beta\), if and
only if both inputs are independently translated and positively scaled
Hermite polynomials.  For every other simple input pair,
\(\vartheta(w_*)\leq(n-2)/n\).
\end{corollary}
\end{samepage}

\begin{proof}
Both scores are centered, so \(w_*\in\mathcal V\).  Their
squared norms are positive by \eqref{eq:radial-pairing}.
Writing \(\gamma=\Omega(\alpha,\beta)\) and
\(J=D\Omega(\alpha,\beta)\), score transport gives
\begin{equation}\label{eq:optimal-score-norms}
 \|w_*\|_2^2=I_\alpha^{-1}+I_\beta^{-1},\qquad
 Jw_*=(I_\alpha^{-1}+I_\beta^{-1})s(\gamma).
\end{equation}
Since \(\Phi_n\) is the same positive multiple of the squared
score norm for all three degree-\(n\) polynomials, Stam equality
is exactly \(\|Jw_*\|_2=\|w_*\|_2\).
The matrix defect and split classification identify this condition with
\(w_*\in L_{\alpha,\beta}\), by \eqref{eq:entropy-order-split}, it is also equivalent to
\(\vartheta(w_*)=1\).
\Cref{thm:information-equality} supplies the Hermite
characterization.  Finally, \eqref{eq:entropy-order-gap} gives
the bound for all other inputs.
\end{proof}

The coefficient \(\vartheta\) remembers the proportion of root
pairs that meet.  In the cubic of
\Cref{ex:cubic-collision-geometry}, the two branches through the
node form one colliding pair, while their gaps to the third root stay
nonzero.  Thus \(\vartheta=2/6=1/3\), and the entropy power
vanishes like \(|\eps|^{2/3}\), as the direct calculation showed.
If all three roots meet, all three pairs contribute and
\(\vartheta=6/6=1\).  The corollary says that the optimally
weighted score motion produces this complete collapse precisely for
the inputs that make Stam sharp.  An equality condition measured at
the simple base fiber is thereby detected by the entropy singularity
of the same curve at its collision fiber.

\section{Further questions}
\label{sec:further-questions}
The entropy criterion detects exact equality through a singular fiber of the curve. To formulate a quantitative version of split rigidity, we return to the simple base fiber, where the Jacobian defect measures the loss of squared Euclidean norm under convolution. We conjecture that this defect controls the squared distance of the input velocity from the split locus, with a lower bound depending only on the degree and the relative input variances, even as individual root gaps tend to zero.

\begin{conjecture}[Stability of finite free contraction]
\label{conj:quantitative-split-rigidity}
For every \(n\geq3\), there exists \(c_n>0\) with the following property.
Let \(\alpha,\beta\in\R^n\) have simple coordinates, put
\(p=P_\alpha\), \(q=P_\beta\), and let
\(J=D\Omega(\alpha,\beta)\).  Write
\(V_p=\Var(p)\) and \(V_q=\Var(q)\).  For every
\(w=u\oplus v\) with \(\sum_i u_i=\sum_jv_j=0\),
\begin{equation}\label{eq:quantitative-split-rigidity}
 \|w\|_2^2-\|Jw\|_2^2
 \geq c_n\,\frac{4V_pV_q}{(V_p+V_q)^2}
          \operatorname{dist}(w,L_{\alpha,\beta})^2.
\end{equation}
The distance is Euclidean, and \(c_n\) is independent of all input
root gaps.
\end{conjecture}

For separately centered motions, distance from \(L_{\alpha,\beta}\) is distance from the radial line \(\operatorname{span}\{\alpha^\circ\oplus\beta^\circ\}\). The kernel classification gives a positive bound on its orthogonal complement for each fixed simple pair, and uniformly on compact subsets of the simple-rooted locus. The conjecture is saying if the input variances stay positive and comparable, individual root gaps may tend to zero without weakening the estimate. The variance factor equals one when the variances agree and tends to zero when one input collapses relative to the other. This reflects the obstruction at a point-mass factor, where convolution merely translates the other configuration and every centered motion of that configuration preserves norm. The specific quadratic decay in the shrinking root scale proposed in \eqref{eq:quantitative-split-rigidity} is also part of the conjecture.

Applying the conjecture to the optimally weighted score direction \(w_*\) from \eqref{eq:optimal-score-direction} would give a quantitative form of Stam:
\begin{equation}\label{eq:conjectural-stam-stability}
\begin{split}
 1-\Phi_n(p\boxplus_nq)
       \left(\frac1{\Phi_n(p)}+\frac1{\Phi_n(q)}\right)
 \geq {}&c_n\,\frac{4V_pV_q}{(V_p+V_q)^2}\\
 &\mathrel{}\times
 \left(1-\frac{\Phi_n(p)^{-1}+\Phi_n(q)^{-1}}{V_p+V_q}\right).
\end{split}
\end{equation}
Indeed, score transport identifies the left-hand side with \((\|w_*\|_2^2-\|Jw_*\|_2^2)/\|w_*\|_2^2\). The radial pairing \eqref{eq:radial-pairing} and orthogonal projection
onto the radial line give
\[
 \frac{\operatorname{dist}(w_*,L_{\alpha,\beta})^2}{\|w_*\|_2^2}
 =1-\frac{\Phi_n(p)^{-1}+\Phi_n(q)^{-1}}{V_p+V_q}.
\]
The same pairing, by Cauchy--Schwarz, gives \(V_p\Phi_n(p)\geq1\), with equality precisely for Hermite inputs. Thus the last factor in \eqref{eq:conjectural-stam-stability} measures the combined departure of the inputs from equality in this variance--Fisher inequality. It vanishes exactly when both inputs are Hermite, with independent translations and scales. Classical information
inequalities have related quantitative stability results \cite{CourtadeFathiPananjady18}, while rigidity of Wasserstein contraction under ordinary convolution is studied in~\cite{FathiGoldmanTsodyks25}. The conjecture here concerns the finite free root Jacobian underlying Stam in~\cite{GVS26}.

In degree three, \Cref{conj:quantitative-split-rigidity} holds with the optimal constant \(c_3=1/2\). To see this, center the inputs and set \(A=\|\alpha\|_2^2\), \(B=\|\beta\|_2^2\), \(R=A+B\), and \(T=\sum_i\alpha_i^3+\sum_i\beta_i^3\).
The output \(\gamma=\Omega(\alpha,\beta)\) has squared norm \(R\) and sum of cubes \(T\).  A separately centered input direction orthogonal to \(\alpha\oplus\beta\) preserves the second moment to first order, so its image under \(J\) lies in the one-dimensional space \(\{\one,\gamma\}^{\perp}\).  Differentiating the sum of cubes equates the pairing of the input velocity with the coordinatewise square vector \(\alpha^2\oplus\beta^2\) to the pairing of its output velocity
with \(\gamma^2\). Both pairings are unchanged if we project these square vectors onto the corresponding complements. The squared norms of the projected vectors are \((A^2+B^2)/6-T^2/R\) on the input side and \(R^2/6-T^2/R\) on the output side. These follow from \(\sum_i\rho_i^4=\tfrac12(\sum_i\rho_i^2)^2\) for any centered cubic root vector \(\rho\).  The squared operator norm is the ratio of these two squared norms. Hence the smallest defect eigenvalue on the centered radial complement is
\begin{equation}\label{eq:cubic-transverse-defect}
 1-\frac{(A^2+B^2)/6-T^2/R}{R^2/6-T^2/R}
 =\frac{2AB}{R^2-6T^2/R}
 \geq\frac{2AB}{R^2}.
\end{equation}
The denominator is positive because the output is simple. Since \(4V_pV_q/(V_p+V_q)^2=4AB/R^2\), this proves the bound on that complement. The radial direction has zero defect, and its image is orthogonal to the images of the complementary directions, so the bound holds for every separately centered motion. It is attained when \(T=0\). For instance, take \(\alpha=\beta=(-1,0,1)\) and \(u=v=(1,-2,1)\).  Then \(J(u\oplus v)=(1,-2,1)\), so the squared input and output speeds are \(12\) and \(6\), respectively, while the direction is orthogonal to the split locus.

The tangent convolution suggests why partial collisions might preserve a positive bound in higher degree.  If neither limiting input is a point mass, let \(m\) and \(k\) be the multiplicities of the input atoms producing a repeated output root.  Its multiplicity \(r=m+k-n\) then satisfies \(r<m,k\).  Both local factors in \Cref{thm:tangent-convolution} therefore undergo genuine degree reduction. Normalized differentiation from degree \(m\) to degree \(r\) multiplies the centered sum of squared roots by \(r(r-1)/(m(m-1))<1\). This gives a possible source of strict contraction inside a cluster.  A proof must still control its interaction with motions of the cluster centers. Analytic degenerations introduce another difficulty, since roots may separate at several scales. The roots \(a\pm\eps^2\), for example, have the same first-order velocity and separate only on the scale \(\eps^2\).

\begin{question}[Successive collision scales]
\label{que:successive-collision-scales}
Let \(\alpha(\eps),\beta(\eps)\in\R^n\) be real analytic root
configurations that are simple for \(\eps\ne0\), and suppose neither
limiting configuration is a point mass.  Can \Cref{thm:tangent-convolution} be extended to successive rescalings of their colliding clusters so as to identify the leading convolution
at each separation scale? Can these descriptions establish \eqref{eq:quantitative-split-rigidity} along such degenerations, with a constant depending only on \(n\)?
\end{question}
The roots outside a cluster must remain in the calculation, since their gaps contribute even to the leading entropy constant in \Cref{cor:collision-information}. The complex geometry raises an independent question. A single cycle in \Cref{thm:real-monodromy} forces absolute irreducibility, but several real components may belong to one complex irreducible curve.
\begin{question}[Complex factors and singularities]
\label{que:complex-geometry}
Fix simple real root vectors \(\alpha,\beta\in\R^n\).  Can the
complex irreducible factors and non-real singularities of \(C_w\)
be characterized directly in terms of the input velocities
\(w=u\oplus v\)?
\end{question}

The rational stratum in \Cref{cor:non-real-discriminant-gap} gives one complete case and beyond it,
the ordered real collision multiplicities alone do not suffice. There is also more to determine on the real normalization.  On each irreducible normalized component, our distinguished projection supplies one degree vector in the separating semigroup of~\cite{KummerShaw20}.

\begin{question}[Other separating morphisms]
\label{que:separating-morphisms}
Which covering-degree vectors occur for separating morphisms from the normalization of an irreducible component of \(C_w\) to \(\mathbf P^1_\R\)?  Can the resulting separating semigroup be described in terms of the input velocities?
\end{question}

Finally, \Cref{sec:local-dynamics} describes repeated self-convolution near Hermite, with a dilation by \(1/\sqrt2\) after each step.  To apply these local estimates to a general normalized configuration, one needs to control the time required to reach that neighborhood.
\begin{question}[Hermite neighborhood]
\label{que:hermite-entry}
For \(n\geq3\), what are the sharp bounds, in terms of \(n\) and \(\delta>0\), for the number of self-convolutions, each followed by dilation by \(1/\sqrt2\), needed to bring every increasingly ordered simple root configuration of mean zero and variance \(n-1\) within Euclidean distance \(\delta\) of the increasing root vector of \(\He_n\)?  In particular, how do these bounds depend on \(n\) when the target neighborhood is chosen so that \Cref{thm:nonlinear-local-contraction} applies?
\end{question}

Understanding this dependence would connect the local root-distance
estimates with the global convergence bounds of~\cite{AP20BerryEsseen}. An estimate uniform in the degree would also need to specify how the target radius depends on \(n\) and whether the initial configurations require further restrictions. The constants \(c_n\) in \Cref{conj:quantitative-split-rigidity} may themselves depend on the degree, so the proposed rigidity estimate leaves this question open.

\appendix

\section{Hermite coupling and the classical root spectrum}
\label{sec:spectrum}

The equality theorem singles out the Hermite configuration.  We now examine
convolution near that configuration.  The tools for this part of the argument
are the Hermite heat equation and the classical spectral theory of its zeros.
They give an explicit description of the root Jacobian, which will then be
used to measure local convergence in Euclidean root distance.

For \(\tau>0\), write
\(H_{n,\tau}(x)=\exp(-\tau\partial_x^2/2)x^n
=\tau^{n/2}\He_n(x/\sqrt\tau)\), and let \(h^{(\tau)}\) be its
increasing root vector.  For \(\tau,\sigma>0\), define
\begin{equation}\label{eq:En-definition}
 E_{n;\tau,\sigma}=D_1\Omega(h^{(\tau)},h^{(\sigma)}),
 \qquad E_n=E_{n;\tau,\tau}.
\end{equation}
The input and output tangent vectors are both represented by their ordered
root coordinates in \(\R^n\).  Simultaneous scaling of the two input root
vectors leaves the Jacobian unchanged, so \(E_n\) is independent of \(\tau\).
Commutativity gives equality of the two partial Jacobians at the diagonal.
The symmetry of either partial Jacobian as a matrix will follow from its
orthogonal modes.

The modes below also occur in the classical inverse-square-gap matrix of
Hermite zeros (see \cite[\S4.1, equations (4.5)--(4.7)]{Sasaki15}) and the
earlier work cited there. We give a direct heat-operator proof and then
identify the precise relation between the two matrices. In particular, this proves the finite-free formulation conjectured in \cite[Conjecture~4.1]{Hashemi26PStam} using that classical structure.

\begin{theorem}[Unequal Hermite coupling spectrum]
\label{thm:exact-spectrum}
Let \(n\geq2\) and \(\tau,\sigma>0\).  For \(1\leq m\leq n\), define
\(v^{(m,\tau)}\in\R^n\) by
\begin{equation}\label{eq:tangent-modes}
 v_i^{(m,\tau)}
 =-\frac{H_{n,\tau}^{(m)}(h_i^{(\tau)})}
         {H_{n,\tau}'(h_i^{(\tau)})}.
\end{equation}
These vectors form an orthogonal basis, and
\begin{equation}\label{eq:En-eigenvalues}
 E_{n;\tau,\sigma}v^{(m,\tau)}
 =\left(\frac{\tau}{\tau+\sigma}\right)^{(m-1)/2}
  v^{(m,\tau)}.
\end{equation}
Consequently, \(E_{n;\tau,\sigma}\) is symmetric positive definite.  Its
operator norm on \(\one^\perp\) is \(\sqrt{\tau/(\tau+\sigma)}\), and
\begin{equation}\label{eq:unequal-En-spectrum}
 \spec(E_{n;\tau,\sigma})
 =\left\{\left(\frac{\tau}{\tau+\sigma}\right)^{(m-1)/2}:
 1\leq m\leq n\right\}.
\end{equation}
In particular,
\begin{equation}\label{eq:En-spectrum}
 \spec(E_n)=\{1,2^{-1/2},2^{-2/2},\ldots,2^{-(n-1)/2}\}.
\end{equation}
\end{theorem}

\begin{proof}
For \(p(x)=\sum_{k=0}^n(-1)^ka_kx^{n-k}\), with \(a_0=1\), introduce
the differential symbol
\begin{equation}\label{eq:differential-symbol}
 \mathcal D_p(z)=\sum_{k=0}^n\frac{(-1)^ka_k}{(n)_k}z^k,
 \qquad (n)_0=1,\quad (n)_k=n(n-1)\cdots(n-k+1).
\end{equation}
Then \(p=\mathcal D_p(\partial_x)x^n\), and the coefficient formula for
finite free convolution gives
\(\mathcal D_{p\boxplus_nq}\equiv\mathcal D_p\mathcal D_q
\pmod{z^{n+1}}\).  The symbol of \(H_{n,\sigma}\) is the degree-\(n\)
truncation of \(\exp(-\sigma z^2/2)\).  Hence convolution with this
polynomial is the heat operator
\begin{equation}\label{eq:heat-operator}
 L_\sigma p=p\boxplus_nH_{n,\sigma}
 =\exp\!\left(-\frac{\sigma}{2}\partial_x^2\right)p.
\end{equation}
In particular, \(L_\sigma H_{n,\tau}=H_{n,\tau+\sigma}\).  If
\(\lambda=\sqrt{(\tau+\sigma)/\tau}\), its roots are
\(\lambda h_i^{(\tau)}\).

Perturb \(H_{n,\tau}\) to
\(H_{n,\tau}+\epsilon H_{n,\tau}^{(m)}\).  Its roots remain simple and
real for sufficiently small real \(\epsilon\).  Implicit differentiation
at the input root \(h_i^{(\tau)}\) gives the velocity
\(v_i^{(m,\tau)}\).  Since \(L_\sigma\) commutes with differentiation,
the output perturbation is \(H_{n,\tau+\sigma}^{(m)}\).  Its root velocity
at \(\lambda h_i^{(\tau)}\) is therefore
\begin{equation}\label{eq:scaled-root-velocity}
 -\frac{H_{n,\tau+\sigma}^{(m)}(\lambda h_i^{(\tau)})}
 {H_{n,\tau+\sigma}'(\lambda h_i^{(\tau)})}
 =\lambda^{1-m}v_i^{(m,\tau)}.
\end{equation}
This proves \eqref{eq:En-eigenvalues}.  If a linear combination of the
vectors \(v^{(m,\tau)}\) vanishes, the corresponding combination of the
polynomials \(H_{n,\tau}^{(m)}\) vanishes at all \(n\) roots of
\(H_{n,\tau}\).  Its degree is at most \(n-1\), so it is zero.
The derivatives have distinct degrees, and hence all coefficients of the
combination vanish.  Thus the modes form a basis.

To check Euclidean orthogonality, use the Appell identity
\(H_{n,\tau}^{(m)}=(n)_mH_{n-m,\tau}\), which gives
\begin{equation}\label{eq:mode-quadrature-form}
 v_i^{(m,\tau)}
 =-\frac{(n)_m}{n}\,
 \frac{H_{n-m,\tau}(h_i^{(\tau)})}
      {H_{n-1,\tau}(h_i^{(\tau)})}.
\end{equation}
For the Gaussian measure with variance \(\tau\), the Gauss--Hermite
quadrature weights at these nodes have the form
\(\omega_i=C_{n,\tau}/H_{n-1,\tau}(h_i^{(\tau)})^2\), with
\(C_{n,\tau}>0\) (see \cite{Szego75}).  When \(m\ne\ell\), the product
\(H_{n-m,\tau}H_{n-\ell,\tau}\) has degree at most \(2n-2\), within the
range of exactness of the quadrature rule.  Hermite orthogonality yields
\begin{equation}\label{eq:mode-orthogonality}
 \sum_{i=1}^n v_i^{(m,\tau)}v_i^{(\ell,\tau)}
 =\frac{(n)_m(n)_\ell}{n^2C_{n,\tau}}
 \sum_{i=1}^n\omega_i
 H_{n-m,\tau}(h_i^{(\tau)})H_{n-\ell,\tau}(h_i^{(\tau)})=0.
\end{equation}
The eigenbasis is therefore orthogonal, and its eigenvalues are positive.
This proves symmetry and positive definiteness.  Finally,
\(v^{(1,\tau)}=-\one\), so the remaining modes span \(\one^\perp\),
where the largest eigenvalue is the \(m=2\) value.
\end{proof}

The connection with the classical root matrix can be read directly from
the heat flow.  Let \(h=h^{(1)}\), and set \(A_n=-Ds(h)\), so that
\begin{equation}\label{eq:classical-hermite-root-matrix}
 (A_n)_{ii}=\sum_{j\ne i}\frac1{(h_i-h_j)^2},
 \qquad (A_n)_{ij}=-\frac1{(h_i-h_j)^2}\quad(i\ne j).
\end{equation}
Suppose that \(p_t=L_tp\) has simple roots \(r_i(t)\).  Differentiation of
\(p_t(r_i(t))=0\) gives \(r_i'(t)=p_t''(r_i(t))/(2p_t'(r_i(t)))\).
The ratio on the right is \(s(r(t))_i\).
Along the Hermite solution \(r(t)=\sqrt{\tau+t}\,h\), the Jacobian
\(U(t)\) of this root flow satisfies
\(U'(t)=-(\tau+t)^{-1}A_nU(t)\) and \(U(0)=I\).  Thus
\begin{equation}\label{eq:hermite-jacobian-exponential}
 E_{n;\tau,\sigma}
 =\exp\!\left(-A_n\log\frac{\tau+\sigma}{\tau}\right).
\end{equation}
After converting between the two Hermite normalizations, the classical
matrix in \cite[\S4.1]{Sasaki15} is \(2A_n\).  Its eigenvectors are the
ratios in \eqref{eq:mode-quadrature-form}, and its eigenvalues are
\(0,1,\ldots,n-1\).  Formula~\eqref{eq:hermite-jacobian-exponential}
therefore identifies the finite-free coupling spectrum with the exponential
of that classical spectrum.

From now on, write \(v^{(m)}=v^{(m,1)}\).

\begin{example}[The four Hermite modes in degree four]
\label{ex:degree-four-modes}
Write \(a=\sqrt{3-\sqrt6}\) and \(b=\sqrt{3+\sqrt6}\).
The polynomial \(H_{4,1}(x)=x^4-6x^2+3\) has increasing root vector
\(h=(-b,-a,a,b)\).  The definition
\(v_i^{(m)}=-H_{4,1}^{(m)}(h_i)/H_{4,1}'(h_i)\), together with
\(h_i^4-6h_i^2+3=0\), gives
\begin{align}\label{eq:degree-four-modes}
 v^{(1)}&=-(1,1,1,1),&
 v^{(2)}&=-h,\notag\\
 v^{(3)}&=\sqrt6(-1,1,1,-1),&
 v^{(4)}&=\sqrt6(1/b,-1/a,1/a,-1/b).
\end{align}
For instance, \(H_{4,1}'(x)=4x(x^2-3)\) and
\(H_{4,1}''(x)=12(x^2-1)\), so
\(v_i^{(2)}=-3(h_i^2-1)/(h_i(h_i^2-3))=-h_i\), where the last equality
uses \(h_i^2(h_i^2-3)=3(h_i^2-1)\).
The third and fourth derivatives are \(24x\) and \(24\), respectively;
substitution at the four roots gives the remaining two vectors.
The first and third vectors are even under reversal of the coordinates,
and the second and fourth are odd.  The mixed inner products vanish by
parity.  The two remaining inner products are
\(\langle v^{(1)},v^{(3)}\rangle=0\) and
\(\langle v^{(2)},v^{(4)}\rangle=\sqrt6(1-1-1+1)=0\).
Their eigenvalues under \(E_4\) are
\(1,2^{-1/2},2^{-1},2^{-3/2}\), respectively.
Translation is the first mode and dilation the second.  When mean and
variance are fixed, the two remaining modes describe changes of shape.
The derivative of normalized self-convolution has multipliers
\(2^{-1/2}\) and \(2^{-1}\) on those modes, as the next section shows.
\end{example}

Commutativity gives the full derivative at the Hermite diagonal as
\(D\Omega(h,h)[u\oplus v]=E_n(u+v)\).  For a unit vector
\(e_m\) in the direction \(v^{(m)}\),
\Cref{thm:exact-spectrum} gives singular value
\(2^{(2-m)/2}\) on \(2^{-1/2}(e_m,e_m)\), while
\(2^{-1/2}(e_m,-e_m)\) lies in the kernel.

\section{Local convergence in root distance}
\label{sec:local-dynamics}

Finite free cumulants already describe normalized self-convolution by an
exact diagonal evolution \cite{AP18}. The root spectrum above permits a
complementary metric statement, with a one-step constant arbitrarily close
to \(2^{-1/2}\) in the ordinary Euclidean norm.  We work throughout on the
manifold of increasingly ordered real roots with mean zero and variance
\(n-1\), namely
\begin{equation}\label{eq:intro-shape-manifold}
 \cM_n=\left\{r\in\R^n:r_1<\cdots<r_n,
 \ \langle r,\one\rangle=0,
 \ \lVert r\rVert_2^2=n(n-1)\right\}.
\end{equation}
For \(P_r(x)=\prod_i(x-r_i)\), define
\begin{equation}\label{eq:intro-T}
 T_n(r)=\operatorname{roots}\!\left(
 2^{-1/2}_*(P_r\boxplus_nP_r)\right),
\end{equation}
with roots listed increasingly. Simplicity of the convolution roots and
additivity of mean and variance give \(T_n(\cM_n)\subset\cM_n\).
The Hermite vector \(h=h^{(1)}\) belongs to \(\cM_n\) and satisfies
\(T_n(h)=h\).

For clarity, the cumulant description can be obtained from the same
differential symbol used above. In the normalization of \cite{AP18}, the
finite free cumulants \(\kappa_m(p)\), \(1\leq m\leq n\), are determined by
the formal identity
\begin{equation}\label{eq:finite-free-cumulant-symbol}
 \log\mathcal D_p(z)
 \equiv-\sum_{m=1}^n\frac{\kappa_m(p)}{m n^{m-1}}z^m
 \pmod{z^{n+1}}.
\end{equation}
The logarithm is defined formally because \(\mathcal D_p(0)=1\).
Multiplication of symbols and dilation give, respectively,
\(\kappa_m(p\boxplus_nq)=\kappa_m(p)+\kappa_m(q)\) and
\(\kappa_m(c_*p)=c^m\kappa_m(p)\).  Consequently,
\begin{equation}\label{eq:exact-cumulant-dynamics}
 \kappa_m(P_{T_n^k(r)})
 =2^{k(1-m/2)}\kappa_m(P_r),
 \qquad 1\leq m\leq n,\quad k\geq0.
\end{equation}
Here \(\kappa_1(p)\) is the root mean and
\(\kappa_2(p)=n\Var(p)/(n-1)\); on \(\cM_n\) they are \(0\) and \(n\).
The Hermite symbol shows that its higher cumulants vanish.  The map from
coefficients to cumulants is triangular with nonzero diagonal, and the map
from simple ordered roots to coefficients has invertible Vandermonde
Jacobian.  Hence \((\kappa_3,\ldots,\kappa_n)\) are analytic local
coordinates on \(\cM_n\) near \(h\).  In these coordinates,
\eqref{eq:exact-cumulant-dynamics} is an exact linearization. This is the
same cumulant mechanism used for the finite free central limit theorem in
\cite[Example~6.1]{AP18}. A related multilinear model realizes additive convolution as multiplication in a squarefree algebra, with cumulants obtained
from its nilpotent logarithm~\cite[\S4]{Sinclair26}.

We next obtain estimates in root distance, retaining the Euclidean
orthogonality of the modes.  The tangent space is
\begin{equation}\label{eq:tangent-space}
 T_h\cM_n=\{u\in\R^n:\langle u,\one\rangle=0,
 \ \langle u,h\rangle=0\}.
\end{equation}
Since \(v^{(1)}=-\one\) and \(v^{(2)}=-h\), this space is spanned by
\(v^{(3)},\ldots,v^{(n)}\).  For \(n=2\), the manifold \(\cM_2\) is a
singleton and there are no shape directions.  We therefore assume
\(n\geq3\) for the remaining statements.

\begin{proposition}[Linearized normalized self-convolution]
\label{prop:linearized-self-convolution}
The analytic extension of \(T_n\) to a neighborhood of \(h\) in the root
chamber has derivative \(DT_n(h)=\sqrt2E_n\).  On \(T_h\cM_n\), its
eigenvalues are
\begin{equation}\label{eq:shape-eigenvalues}
 2^{-1/2},2^{-2/2},\ldots,2^{-(n-2)/2}.
\end{equation}
Its Euclidean operator norm on this tangent space is \(2^{-1/2}\).
\end{proposition}

\begin{proof}
The two identical inputs in \eqref{eq:intro-T} contribute the partial
derivatives \(E_n\) and \(E_n\).  The output dilation contributes
\(2^{-1/2}\), giving
\(DT_n(h)=2^{-1/2}(E_n+E_n)=\sqrt2E_n\).  Restrict
\Cref{thm:exact-spectrum} to the modes \(m=3,\ldots,n\).
\end{proof}

\begin{theorem}[Nonlinear local contraction]
\label{thm:nonlinear-local-contraction}
For every \(\eps>0\), there is \(\delta=\delta(n,\eps)>0\) such that
\(\{r\in\cM_n:\lVert r-h\rVert_2<\delta\}\) is forward invariant and
\begin{equation}\label{eq:one-step-local-contraction}
 \lVert T_n(r)-h\rVert_2
 \leq(2^{-1/2}+\eps)\lVert r-h\rVert_2
\end{equation}
on that neighborhood.  In particular,
\begin{equation}\label{eq:iterated-local-contraction}
 \lVert T_n^k(r)-h\rVert_2
 \leq(2^{-1/2}+\eps)^k\lVert r-h\rVert_2
 \qquad(k\geq0).
\end{equation}
\end{theorem}

\begin{proof}
The output roots remain simple near \(h\), so \(T_n\) extends to a real
analytic map on a neighborhood of \(h\) in the root chamber.  There are
\(C,R_0>0\) such that
\begin{equation}\label{eq:T-taylor}
 T_n(h+d)-h=DT_n(h)d+R(d),
 \qquad \lVert R(d)\rVert_2\leq C\lVert d\rVert_2^2
\end{equation}
when \(\lVert d\rVert_2<R_0\).
For \(h+d\in\cM_n\), the two constraints give
\(\langle d,\one\rangle=0\) and
\(2\langle h,d\rangle+\lVert d\rVert_2^2=0\).
Write \(d=u+ah\), where \(u\in T_h\cM_n\).  Then
\(a=-\lVert d\rVert_2^2/(2\lVert h\rVert_2^2)\).
Thus the radial component of a constrained displacement is quadratic in
its size.  The radial mode has multiplier \(1\) under \(DT_n(h)\), while
the tangent modes have norm at most \(q=2^{-1/2}\).  It follows that
\begin{align}\label{eq:nonlinear-bound}
 \lVert T_n(h+d)-h\rVert_2
 &\leq q\lVert u\rVert_2+|a|\lVert h\rVert_2
       +C\lVert d\rVert_2^2\notag\\
 &\leq q\lVert d\rVert_2+K\lVert d\rVert_2^2,
 \qquad K=C+\frac1{2\lVert h\rVert_2}.
\end{align}
Set \(\eps_0=\min\{\eps,(1-q)/2\}\), and choose \(\delta<R_0\) with
\(K\delta\leq\eps_0\).  The estimate then holds with the factor
\(q+\eps_0<1\), so the neighborhood is forward invariant.  Iteration
proves the asserted bounds, first with \(\eps_0\) and then with \(\eps\).
\end{proof}

The quadratic remainder also permits the exact dyadic rate.  Its accumulated
effect is a bounded prefactor, because the preceding contraction makes the
sum of the successive distances finite.

\begin{corollary}[Exact dyadic rate in root distance]
\label{cor:exact-dyadic-root-rate}
There are \(\delta_n,B_n>0\) such that, for every \(r\in\cM_n\) with
\(e_0=\lVert r-h\rVert_2<\delta_n\),
\begin{equation}\label{eq:exact-dyadic-root-rate}
 \lVert T_n^k(r)-h\rVert_2
 \leq \exp(B_ne_0)\,2^{-k/2}e_0
 \qquad(k\geq0).
\end{equation}
Thus the prefactor tends to one as the initial root vector tends to \(h\).
\end{corollary}

\begin{proof}
Use the constants \(R_0,K\) from \eqref{eq:nonlinear-bound}, put
\(q=2^{-1/2}\) and \(\rho=(1+q)/2\), and choose \(\delta_n<R_0\) so
that \(K\delta_n\leq\rho-q\).  If
\(e_j=\lVert T_n^j(r)-h\rVert_2\), then \(e_j\leq\rho^je_0\) and
\(e_{j+1}\leq qe_j(1+Ke_j/q)\).  Multiplying these inequalities and
using \(1+x\leq e^x\) for \(x\geq0\) gives
\begin{align*}
 e_k
 &\leq q^ke_0\prod_{j=0}^{k-1}\left(1+\frac Kq e_j\right)\\
 &\leq q^ke_0\exp\!\left(\frac Kq\sum_{j=0}^{k-1}e_j\right)
 \leq q^ke_0\exp\!\left(\frac{K e_0}{q(1-\rho)}\right).
\end{align*}
Taking \(B_n=K/[q(1-\rho)]\) proves the claim.
\end{proof}

At the fixed point, the derivative cocycle is \(DT_n(h)^k\).
Taking logarithms of its positive eigenvalues in
\eqref{eq:shape-eigenvalues} gives the Lyapunov exponents
\(\Lambda_j=-(j/2)\log2\), \(1\leq j\leq n-2\), on
\(T_h\cM_n\).

These statements concern the ordered normalized root manifold \(\cM_n\).
They describe both the infinitesimal contraction and its effect on nearby
root vectors.  The largest exponent agrees with the exact evolution of the
third finite free cumulant in \eqref{eq:exact-cumulant-dynamics}.
Arizmendi and Perales proved a global \(m^{-1/2}\) Berry--Esseen estimate
in L\'evy distance, with a constant depending on the degree
\cite{AP20BerryEsseen}.  \Cref{cor:exact-dyadic-root-rate} gives
the corresponding rate in local Euclidean root distance for \(m=2^k\).
Its near-identity prefactor and one-step estimate come from the orthogonal
root geometry.  The neighborhood size and the constants here depend on
\(n\); no estimate uniform in the degree is asserted.

\section{Geometric interpretations}\label{sec:discriminant-geometry}

The split classification and the Hermite equality theorem have two direct
algebraic interpretations.  We record them here to clarify the geometry
of the classified loci.  Neither interpretation is needed for the
contraction estimates in the main text.

\subsection{The split incidence}
The variety of products of linear forms is the classical split
variety studied, for example, by Arrondo and Bernardi \cite{ArrondoBernardi11}. The proposition below compares its dimension with the dimension of its real preimage in the finite free family.

\begin{proposition}[Dimension of the real split incidence]
\label{prop:split-excess}
The variety of products of linear forms
\begin{equation}\label{eq:split-chow-variety}
 \operatorname{Split}_n(\C)
 =\{[\ell_1\cdots\ell_n]:\ell_i\in\C[x,y,t]_1\setminus\{0\}\}
 \subset\mathbf P(\operatorname{Sym}^n\C^3)
\end{equation}
is irreducible of dimension \(2n\) and codimension \(\binom n2\).
For fixed simple real \(\alpha,\beta\), let
\(\Theta_{\alpha,\beta}(w)=[h_w]\), and let
\(\operatorname{Split}_n(\R)\) denote products of real linear forms.
Then
\begin{equation}\label{eq:real-split-fiber}
 \Theta_{\alpha,\beta}^{-1}(\operatorname{Split}_n(\R))
 =\operatorname{span}\{\one\oplus0,\ 0\oplus\one,\
                         \alpha^\circ\oplus\beta^\circ\}.
\end{equation}
Its dimension is three.  The difference from the transverse expected
dimension \(2n-\binom n2\) is \(\binom{n-2}{2}\).
\end{proposition}

\begin{proof}
Multiplication gives a morphism from \((\mathbf P^{2\vee})^n\) onto
\(\operatorname{Split}_n(\C)\).  On the open set of pairwise
nonproportional factors, unique factorization leaves only the finite
ambiguity of their order.  The image is therefore irreducible of
dimension \(2n\), and its codimension is
\(\binom{n+2}{2}-1-2n=\binom n2\).
The inverse-image identity is \Cref{thm:intro-split-locus}; its three
generators are independent because both input variances are positive.
Subtracting the transverse expected dimension gives
\(3-2n+\binom n2=\binom{n-2}{2}\).
\end{proof}

This dimension comparison concerns the real preimage classified by the
theorem.  It does not determine the complex inverse-image scheme or an
intersection multiplicity.  Opposite input translations give the
ineffective direction \(\one\oplus(-\one)\) of
\(\Theta_{\alpha,\beta}\).  Removing this direction lowers both
dimensions by one and leaves the comparison unchanged.

\subsection{The Hermite critical scheme}

The radial score equation also makes sense over \(\C\).  Let
\begin{equation}\label{eq:intro-discriminant-complement}
 U_n=\{z\in\C^n:\Delta(z)\ne0\},
 \qquad \Delta(z)=\prod_{i<j}(z_i-z_j)^2.
\end{equation}
On this ordered discriminant complement, the algebraic one-form
\begin{equation}\label{eq:intro-logarithmic-form}
 \omega_n=d\log\Delta-\sum_i z_i\,dz_i
          =\sum_i\bigl(2s(z)_i-z_i\bigr)\,dz_i
\end{equation}
is regular.  The next proposition adds a reducedness statement to the
classical electrostatic characterization of Hermite zeros
\cite{Szego75}.

\begin{proposition}[The Hermite critical scheme]
\label{thm:intro-hermite-critical-scheme}
Let \(h_1<\cdots<h_n\) be the roots of \(\He_n\).  The zero scheme of
\(\omega_n\) on \(U_n\) is reduced and consists of the \(n!\)
orderings of \(h\).  Every zero satisfies
\(\sum_i z_i=0\) and \(\sum_i z_i^2=n(n-1)\).
\end{proposition}

\begin{proof}
At a zero, \(s(z)_i=z_i/2\).  Summing and pairing with \(z\) give
\begin{equation}\label{eq:critical-normalizations}
 \sum_i z_i=0,
 \qquad \frac12\sum_i z_i^2
 =\sum_i z_i s(z)_i
 =\sum_{i<j}\frac{z_i-z_j}{z_i-z_j}
 =\binom n2.
\end{equation}
Set \(P(x)=\prod_i(x-z_i)\).  At its roots,
\(P''(z_i)/(2P'(z_i))=s(z)_i=z_i/2\).  Thus
\(P''-xP'+nP\) has degree at most \(n-1\) and vanishes at \(n\)
distinct points.  It is zero.  The unique monic degree-\(n\) solution
of this differential equation is \(\He_n\), so every closed point of
the zero scheme is a Hermite ordering, and all orderings occur.

To check reducedness, put \(F_i(z)=2s(z)_i-z_i\).  At a Hermite
ordering \(h\), every nonzero \(u\in\R^n\) satisfies
\begin{equation}\label{eq:critical-jacobian-negative}
 u^{\mathsf T}DF(h)u
 =-2\sum_{i<j}\frac{(u_i-u_j)^2}{(h_i-h_j)^2}
  -\|u\|_2^2<0.
\end{equation}
Hence the Jacobian is invertible over \(\R\), and therefore over
\(\C\).  The Jacobian criterion makes every zero a reduced isolated
point, as asserted.
\end{proof}

\subsection{Affine shapes and the information retained by roots}

The affine group acts freely on ordered configurations in \(U_n\).
Adding a marked point at infinity identifies the quotient with
\(M_{0,n+1}\).  The further quotient
\(\mathfrak X_n=M_{0,n+1}/S_n\) forgets the ordering of the finite
points.  By \Cref{thm:information-equality}, the image of the real
simple-rooted entropy-power equality locus in
\(\mathfrak X_n\times\mathfrak X_n\) is the single pair
\(([\He_n],[\He_n])\).  This is a reformulation of the equality
classification.  The preceding proposition concerns the normalized
ordered critical scheme before taking this quotient.

For the quantitative questions, it matters which quotient is used.
Increasingly ordered real roots, centered and normalized by a positive
scale, retain the distinction between a configuration and its reflection.
The affine moduli quotient identifies them.  Separate reflections of two
inputs can change their convolution, so two moduli points and a relative
variance do not, in general, determine the Stam deficit.

\begin{example}[Reflection changes the deficit]\label{ex:reflection-deficit}
Let \(f_c(x)=x^3-3x+c\), where \(|c|<2\).  Its discriminant is
\(108-27c^2>0\), so its roots are simple and real.  Their mean is zero
and their variance is two.  Reflection takes \(f_c\) to \(f_{-c}\),
and hence these polynomials represent the same point of
\(\mathfrak X_3\).  The coefficient formula nevertheless gives
\begin{equation}\label{eq:reflection-convolutions}
 f_1\boxplus_3f_1=x^3-6x+2,
 \qquad f_1\boxplus_3f_{-1}=x^3-6x.
\end{equation}
For a simple centered cubic \(p(x)=x^3+ax+b\), write
\(\Delta=-4a^3-27b^2\).  The computation in
\Cref{ex:cubic-collision-geometry} gives
\(\Phi_3(p)=6a^2/\Delta\).  Thus each input in the two pairs has
reciprocal Fisher information \(3/2\), whereas the two outputs in
\eqref{eq:reflection-convolutions} have reciprocal informations
\(7/2\) and \(4\).  Their Stam deficits are consequently
\(1/2\) and \(1\), despite identical individual moduli points and
variances.

The same family illustrates the distinction for tangent dynamics. Variance-normalized self-convolution sends \(f_c\) to
\(f_{c/\sqrt2}\).  Its multiplier in the real root-shape coordinate \(c\) is \(1/\sqrt2\).  A coordinate on the coarse quotient near the Hermite point is \(c^2\), whose multiplier is instead \(1/2\). The tangent estimates in \Cref{sec:local-dynamics} are statements on the real root manifold \(\cM_n\).
\end{example}

\bibliographystyle{alpha}
\bibliography{references}

\end{document}